\documentclass[11pt]{amsart}
\usepackage{amsfonts}
\usepackage{amsthm}
\usepackage{graphicx}

\usepackage[text={425pt,650pt},centering]{geometry}
\usepackage{color}
\usepackage{esint,amssymb}
\usepackage{tikz}

\newtheorem{theorem}{Theorem}[section]
\newtheorem{lemma}[theorem]{Lemma}
\newtheorem{proposition}[theorem]{Proposition}
\newtheorem{corollary}[theorem]{Corollary}
\newtheorem{definition}[theorem]{Definition}
\newtheorem{remark}[theorem]{Remark}

\newcommand{\negint}{{\int\negthickspace\negthickspace\negthickspace\negthickspace-}}

\def\R{\mathbb R}

\newcommand{\ep}{\varepsilon}

\title 
[Free boundary regularity in obstacle problems]{Free boundary regularity
in obstacle problems}

\author
[A. {Figalli}]
{Alessio Figalli}
\address{ETH Z\"urich, Mathematics Department, R\"amistrasse 101, 8092 Z\"urich, Switzerland.}
\thanks{A.F. is supported by ERC Grant ``Regularity and Stability in Partial Differential Equations (RSPDE)''. A.F. is thankful to Yash Jhaveri for useful comments on a preliminary version of this manuscript.}
\email{alessio.figalli@math.ethz.ch}


\begin{document}

\begin{abstract}
  These notes record and expand the lectures for the ``Journ\'ees \'Equations aux D\'eriv\'ees Partielles 2018'' held by
the author during the week of June 11-15, 2018. 
The aim is to give a overview of the classical theory for the obstacle problem, and then present some recent developments on the regularity of the free boundary. 
\end{abstract}

%

\maketitle
\tableofcontents

\section{Introduction: the obstacle problem}
The classical obstacle problem aims to describe the shape of an elastic membrane lying above an obstacle.

Mathematically, this problem can be described as follows:
given a domain $\Omega\subset \R^n$,
a function $\varphi:\Omega\to \R$ (the obstacle),
and a function $f:\partial\Omega\to \R$ satisfying $f\geq \varphi|_{\partial\Omega}$ (the boundary condition),
one wants to minimize the Dirichlet integral among all functions that coincide with $f$ on $\partial \Omega$ and lie above the obstacle $\varphi$ (see Figure \ref{Pic29b}).
 \begin{figure}[ht]
\includegraphics[scale=0.23]{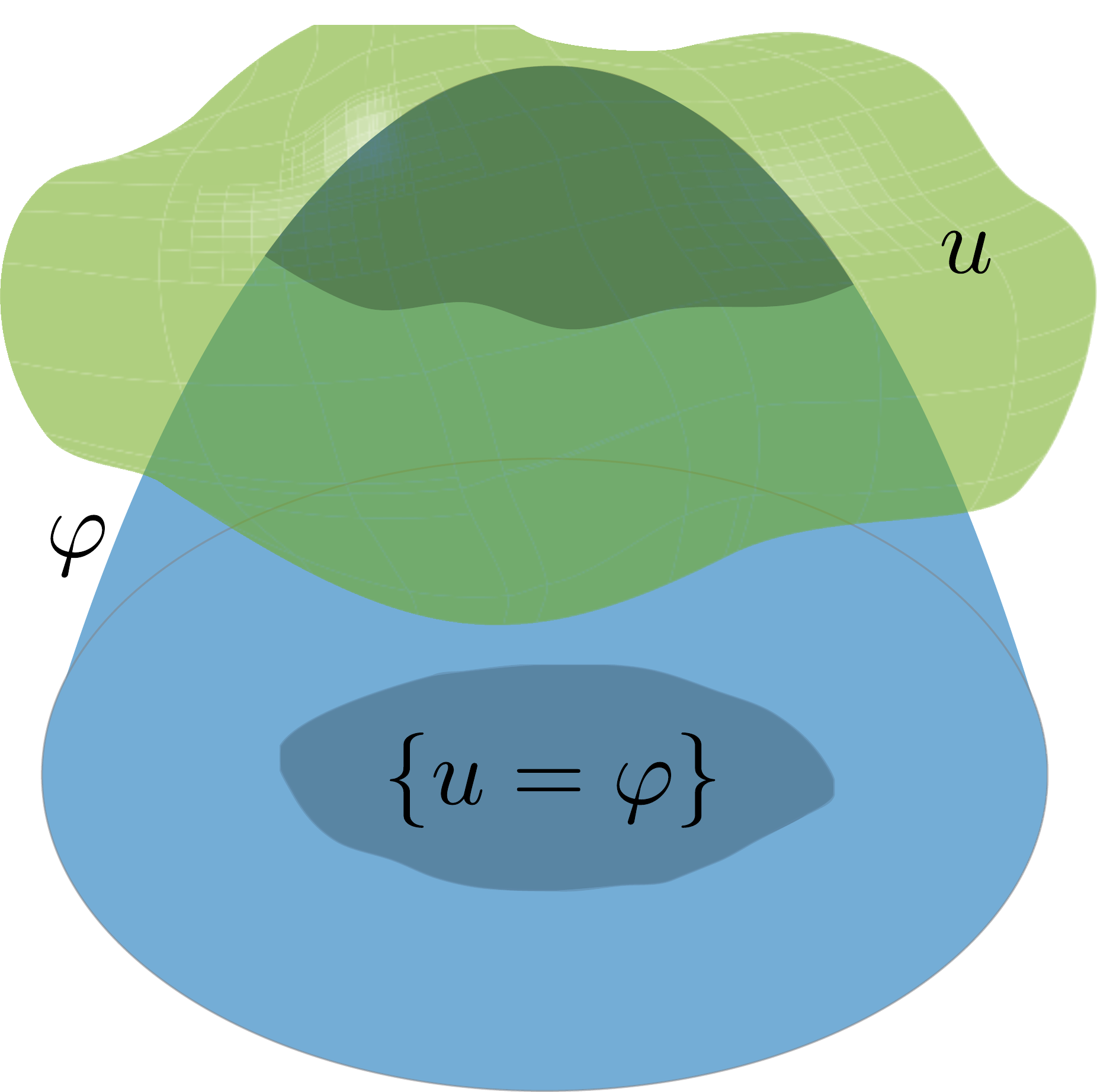} \caption{An elastic membrane lying above an obstacle.} 
  \label{Pic29b}
 \end{figure}
Hence, one is lead to minimize
\begin{equation}
\label{eq:min}
\min_v \biggl\{\int_\Omega\frac{|\nabla v|^2}{2}\,:\,v|_{\partial\Omega}=f,\,v \geq \varphi\biggr\},
\end{equation}
where the Dirichlet integral $\int_\Omega\frac{|\nabla v|^2}{2}$ represents the elastic energy of  the membrane corresponding to  the graph of $v$.

The main goal consists in understanding the regularity properties of the minimizer, as well as the structure of the contact set between the minimizer and the obstacle.
In the next section we shall describe all these questions in detail.

\section{The functional setting: existence and uniqueness}
The proof of existence and uniqueness of minimizers is very similar to the one for the classical Dirichlet problem without obstacle.
Since our focus will be in understanding how the obstacle influence the solution, instead of working under minimal regularity assumptions, we shall assume that all the data are smooth in order to emphasize the main ideas.
Hence, we assume that $\Omega$ is a bounded domain of class $C^1,$ $\varphi:\overline\Omega\to \R$ is of class $C^1$, and that $f:\partial\Omega\to \R$ is $C^1$ as well.
Note that, because of these hypotheses, one can extend $f$ to a $C^1$ function
$F:\overline\Omega\to \R$.
Recall that, by assumption, $\varphi|_{\partial\Omega}\leq f$.

Under these hypotheses, we can show existence and uniqueness of minimizers in $W^{1,2}(\Omega)$.
\begin{proposition}
\label{prop:exist}
There exists a unique minimizer  $u\in W^{1,2}(\Omega)$ for the minimization problem \eqref{eq:min}.
\end{proposition}
\begin{proof}
The existence follows by observing that \eqref{eq:min} corresponds to minimizing the Dirichlet functional among all functions $v$ that belong to the convex set
$$
\mathcal K_\varphi:=\{v \in W^{1,2}(\Omega)\,:\,v|_{\partial\Omega}=f,\,v \geq \varphi\},
$$
where the relation $v|_{\partial\Omega}=f$ must be intended in the sense of traces of Sobolev functions.
Note that $\mathcal K_\varphi$ is closed in the strong $W^{1,2}$ topology, hence it is also closed for the weak $W^{1,2}$ topology (as a consequence of Hanh-Banach separation theorem).
Thus, to find a minimizer  of \eqref{eq:min} it suffices to argue as in the case of the classical Dirichlet problem.

More precisely, consider a minimizing sequence $\{v_k\}_{k\geq 1}$, namely a sequence of functions $\{v_k\}_{k\geq 1} \subset \mathcal K_\varphi$ such that
\begin{equation}
\label{eq:min seq}
\int_\Omega \frac{|\nabla v_k|^2}2  \to \alpha:=\inf_{v\in \mathcal K_\varphi} \int_\Omega \frac{|\nabla v|^2}2 .
\end{equation}
Also, fix a Lipschitz function $V\in \mathcal K_\varphi$ (for instance, one can define $V:=\max\{\varphi,F\}$, where $F$ is a $C^1$ extension of $f$ as explained above).

Note that $\alpha $ is finite since  $\alpha \leq \int_\Omega \frac{|\nabla V|^2}2 <\infty$.
Thus, thanks to \eqref{eq:min seq}, 
there exists $k_0>0$ such that
\begin{equation}
\label{eq:Dir k}
\int_\Omega \frac{|\nabla v_k|^2}2  \leq \alpha+1\qquad \forall\, k \geq k_0.
\end{equation}
Furthermore, by Poincar\'e inequality applied to the function $v_k-V \in W^{1,2}_0(\Omega)$,
\begin{equation}
\label{eq:L2 k}
\|v_k-V\|_{L^2(\Omega)}\leq C_\Omega
\|\nabla v_k-\nabla V\|_{L^2(\Omega)} 
\end{equation}
Hence, combining \eqref{eq:Dir k} and \eqref{eq:L2 k}, for all $k \geq k_0$ we get
\begin{align*}
\|v_k\|_{L^2(\Omega)}+\|\nabla v_k\|_{L^2(\Omega)} &\leq 
\|v_k-V\|_{L^2(\Omega)}+ \|V\|_{L^2(\Omega)}+\|\nabla v_k\|_{L^2(\Omega)}\\
&\leq 
C_\Omega
\|\nabla v_k-\nabla V\|_{L^2(\Omega)}+ \|V\|_{L^2(\Omega)}+\|\nabla v_k\|_{L^2(\Omega)}\\
&\leq 
(C_\Omega+1)
\|\nabla v_k\|_{L^2(\Omega)}+
\|V\|_{L^2(\Omega)}+ C_\Omega\|\nabla V\|_{L^2(\Omega)}\\
&\leq 
(C_\Omega+1)
\sqrt{2(\alpha+1)}+
\|V\|_{L^2(\Omega)}+ C_\Omega\|\nabla V\|_{L^2(\Omega)}.
\end{align*}
This proves that the functions $v_k$ are uniformly bounded in $W^{1,2}(\Omega)$, hence there exists a subsequence $v_{k_j}$ that converges weakly to a function $u \in W^{1,2}(\Omega)$. 
Since $v_{k_j} \in \mathcal K_\varphi$ and the set $\mathcal K_\varphi$ is weakly closed (by the discussion above), it follows that $u \in \mathcal K_\varphi$.
Finally, since the $L^2$ norm is lower semicontinuous under weak convergence,
$$
\int_\Omega \frac{|\nabla u|^2}2  \leq \liminf_{j\to \infty}\int_\Omega \frac{|\nabla v_{k_j}|^2}2 =\alpha,
$$
which proves that $u$ is a minimizer.

For the uniqueness, it suffices to observe that if $u_1,u_2 \in \mathcal K_\varphi$ then
$$
\int_\Omega\biggl|\frac{\nabla u_1+\nabla u_2}{2}\biggr|^2  \leq \frac12\biggl(\int_\Omega|\nabla u_1|^2 +\int_\Omega |\nabla u_2|^2 \biggr),
$$
with equality if and only if $\nabla u_1\equiv \nabla u_2$. In particular, if $u_1$ and $u_2$ are two minimizers then equality must hold (otherwise $\frac{u_1+u_2}2$ would have strictly less Dirichlet energy), therefore
$$
\nabla (u_1-u_2)= 0\text{ in }\Omega,\qquad u_1-u_2=0 \text{ on }\partial\Omega.
$$
By Poincar\'e inequality this implies that $u_1-u_2=0$, as desired.
\end{proof}

\section{Euler-Lagrange equation and consequences}
It is a well-known fact that minimizers of the Dirichlet energy are harmonic. However, in our case, the presence of the obstacle $\varphi$ plays an important role.
\begin{proposition}
\label{prop:EL}
Let $u:\Omega \to \R$ be the minimizer of \eqref{eq:min}. Then $\Delta u \leq 0$ inside $\Omega.$
\end{proposition}
\begin{proof}
Let $\psi \in C^\infty_c(\Omega)$ be nonnegative, and for $\epsilon>0$ consider the function
$u_\epsilon:=u+\epsilon\psi$. Since $\psi\geq 0$ it follows that $u_\epsilon\geq u \geq \varphi$.
Also, because $\psi$ is compactly supported in $\Omega$, $u_\epsilon|_{\partial\Omega}=u|_{\partial\Omega}=f.$
This shows that $u_\epsilon$ is admissible in the minimization problem \eqref{eq:min}, thus (by the minimality of $u$)
\begin{align*}
\int_{\Omega}\frac{|\nabla u|^2}{2} &\leq \int_{\Omega}\frac{|\nabla u_\epsilon|^2}2 = \int_{\Omega}\frac{|\nabla u+\epsilon\nabla\psi|^2}{2}\\
&=\int_{\Omega}\frac{|\nabla u|^2}{2} + \epsilon\int_\Omega \nabla u\cdot \nabla \psi +\epsilon^2\int_{\Omega}\frac{|\nabla \psi|^2}{2}.
\end{align*}
Simplifying the term $\int_{\Omega}\frac{|\nabla u|^2}{2}$ from the first and last expression, and then dividing by $\epsilon$, we obtain
$$
0\leq \int_\Omega \nabla u\cdot \nabla \psi +\epsilon\int_{\Omega}\frac{|\nabla \psi|^2}{2}.
$$
Thus, letting $\epsilon\to 0^+$ we get
$$
0\leq \int_\Omega \nabla u\cdot \nabla \psi
=-\int_\Omega \Delta u\, \psi,
$$
where the last equality must be intended in the sense of distribution. Since $\psi$ is an arbitrary nonnegative smooth function, the inequality above implies that $-\Delta u\geq 0$, as desired.
\end{proof}

As a consequence of the previous result, we can find a ``nice'' representative for $u$.

\begin{corollary}
Let $u:\Omega \to \R$ be the minimizer of \eqref{eq:min}. Then, up to changing $u$ in a set of measure zero, $u$ is lower semicontinuous.
\end{corollary}
\begin{proof}
By the mean value formula for superharmonic functions
\begin{equation}
\label{eq:super harm u}
\Delta u \leq 0\quad \text{in }\Omega\qquad \Rightarrow\qquad 
r\mapsto \negint_{B_r(x)}u\quad  \text{is decreasing on $(0,R_x)$},
\end{equation}
where $R_x:={\rm dist}(x,\partial\Omega)$.
Define the function
$$
\hat u(x):=\lim_{r\to 0}\negint_{B_r(x)}u\qquad \forall\,x\in\Omega
$$
(note that the limit exists thanks to \eqref{eq:super harm u}).
Then $\hat u(x)=u(x)$ whenever $x$ is a Lebesgue point for $u$, therefore $\hat u=u$ a.e.

Also, if $x_k\to x_\infty$ then, using \eqref{eq:super harm u} again, we get
$$
\negint_{B_r(x_\infty)}u =\lim_{k\to \infty}\negint_{B_r(x_k)}u \leq \liminf_{k\to \infty}\hat u(x_k)\qquad \forall\,r \in (0,R_{x_\infty}/2),
$$
where the first equality follows by the Lebesgue dominated convergence theorem noticing that
$$
\negint_{B_r(x_k)}u=\frac{1}{|B_r|}\int_{\Omega}u\chi_{B_r(x_k)}\qquad \text{and}\qquad u\chi_{B_r(x_k)}\to u\chi_{B_r(x_\infty)} \text{ a.e.}
$$
Letting $r\to 0$ we obtain
$$
\hat u(x_\infty)\leq \liminf_{k\to \infty}\hat u(x_k),
$$
as desired.
\end{proof}
From now on we will implicitly assume that $u$ coincides with its lower semicontinuous representative $\hat u$. In particular, $u$ is pointwise defined at every point.

An important consequence of the previous result is the following:
\begin{corollary}
Let $u:\Omega \to \R$ be the minimizer of \eqref{eq:min}. Then
$$
\text{the set }\quad\{u>\varphi\}\cap \Omega \quad \text{ is open}.
$$
\end{corollary}
\begin{proof}
This is a direct consequence of the fact that, since $u$ is lower semicontinuous and $\varphi$ is $C^1$ (hence continuous), then also $u -\varphi$ is lower semicontinuous inside $\Omega$. 
Hence, if $x_k \to x_\infty\in \Omega$ and $x_k \in \{u- \varphi \leq 0\}\cap \Omega$, then
$$
(u-\varphi)(x_\infty)\leq 
 \liminf_{k\to \infty}(u-\varphi)(x_k) \leq 0.
$$ 
This proves that the set $\{u\leq \varphi\}$ is closed inside $\Omega$, thus
the set $\{u>\varphi\}\cap \Omega$ is open.
\end{proof}

Thanks to the previous result, we can now show that $u$ is harmonic away from the contact set $\{u=\varphi\}$.

\begin{corollary}
\label{cor:harm}
Let $u:\Omega \to \R$ be the minimizer of \eqref{eq:min}. Then
$$
\Delta u=0\qquad \text{inside }\{u>\varphi\}\cap \Omega. 
$$
\end{corollary}
\begin{proof}
Fix  a ball $B_r(x_0) \subset \subset \{u>\varphi\}\cap \Omega$, and consider $\psi \in C^\infty_c(B_r(x_0))$.
For $\epsilon>0$, define the function
$u_\epsilon:=u+\epsilon\psi$.

Since $B_r(x_0) \subset \subset \{u>\varphi\}\cap \Omega$
and $u-\varphi$ is lower semicontinuous, it must attain a positive minimum inside $\overline{B_r(x_0)}$, namely
$$
\min_{\overline{B_r(x_0)}}(u-\varphi)=:c_0>0
$$
Hence, if $\epsilon>0$ is small enough (the smallness depending on $\|\psi\|_{L^\infty(B_r(x_0))}$), it follows that
\begin{multline*}
u_\epsilon(x) \geq u(x)-\epsilon\|\psi\|_{L^\infty(B_r(x_0))}
 \geq \varphi(x)+c_0-\epsilon\|\psi\|_{L^\infty(B_r(x_0))} \geq \varphi(x)\qquad \forall\, x \in B_r(x_0).
\end{multline*}
Since $u_\epsilon=u$ outside $B_r(x_0)$, this shows that $u_\epsilon$ is admissible in the minimization problem \eqref{eq:min}.
Thus, arguing as in the proof of Proposition \ref{prop:EL} we obtain
\begin{equation}
\label{eq:EL2}
0\leq \int_\Omega \nabla u\cdot \nabla \psi
=-\int_\Omega \Delta u\, \psi,
\end{equation}
where the last equality must be intended in the sense of distribution. 
Since now $\psi$ is arbitrary, we can replace $\psi$ by $-\psi$ in \eqref{eq:EL2} to obtain
$$
0\leq -\int_\Omega \Delta u\, (-\psi).
$$
Combining this inequality with \eqref{eq:EL2} we conclude that
$$
\int_\Omega \Delta u\, \psi=0\qquad \forall\,\psi \in C^\infty_c(B_r(x_0)),
$$
which implies that $\Delta u=0$ inside $B_r(x_0)$.

Since $B_r(x_0)$ was an arbitrary ball compactly contained inside $\{u>\varphi\}\cap \Omega$, this proves that $\Delta u=0$ inside $\{u>\varphi\}\cap \Omega$, as desired.
\end{proof}

\section{Optimal regularity}
In this section we prove the optimal regularity of $u$, namely $u \in C^{1,1}_{\rm loc}(\Omega)$. This result was first obtained in \cite{Fre72} (see also \cite{BK74,C98}).
As we shall discuss later, the $C^{1,1}$ regularity of $u$ is optimal.

Consider the function
$$
v:=u-\varphi, \qquad v:\Omega\to \R.
$$
Note that, thanks to Proposition \ref{prop:EL} and Corollary \ref{cor:harm}, 
\begin{equation}
\label{eq:v}
v\geq 0,\qquad \Delta v \leq -\Delta \varphi \quad \text{in $\Omega$},\qquad \Delta v=-\Delta \varphi\quad  \text{in $\{v>0\}\cap \Omega$}.
\end{equation}

\begin{theorem}
\label{thm:C11}
Let $v=u-\varphi$, where $u$ is the minimizer of \eqref{eq:min}. Assume that $\varphi \in C^{1,1}(\Omega)$.
Then
$$
v \in C^{1,1}_{\rm loc}(\Omega).
$$
\end{theorem}
Of course, since $\varphi \in C^{1,1}(\Omega)$ by assumption, the  theorem above implies that also $u$ belongs to $C^{1,1}_{\rm loc}(\Omega)$.

To prove the result, we first show that $v$ grows at most quadratically away from the contact set.
\begin{lemma}
\label{lem:quadr}
Let $v=u-\varphi$, where $u$ is the minimizer of \eqref{eq:min}. Assume that $\Delta\varphi \in L^\infty(\Omega)$. 
Then there exists a constant $C'$, depending only on $n$ and $\|\Delta\varphi\|_{L^\infty(\Omega)}$, such that
\begin{equation}
\label{eq:r2 above}
0\leq v(x)\leq C'{\rm dist}(x,\partial \{v>0\})^2
\end{equation}
for all points $x\in\{v>0\}\cap\Omega$ such that ${\rm dist}(x,\partial\{v>0\})\leq \frac13{\rm dist}(x,\partial\Omega).$
\end{lemma}
\begin{proof}
To simplify the notation set $C_0:=\|\Delta\varphi\|_{L^\infty(\Omega)}$.
We begin by recalling the mean value formula for (super/sub)harmonic functions:
\begin{equation}
\label{eq:super harm}
\Delta w \leq 0\quad \text{in }B_R(x)\qquad \Rightarrow\qquad 
w(x)\geq \negint_{B_r(x)}w\qquad \forall\,r\in (0,R),
\end{equation}
\begin{equation}
\label{eq:sub harm}
\Delta w \geq  0\quad \text{in }B_R(x)\qquad \Rightarrow\qquad 
w(x)\leq  \negint_{B_r(x)}w\qquad \forall\,r\in (0,R).
\end{equation}
Since $\Delta v \leq -\Delta\varphi$ inside $\Omega$ (see \eqref{eq:v}), we can apply \eqref{eq:super harm} to the function $w(y):=v(y)-C_0\frac{|x-y|^2}{2n}$ to obtain
\begin{multline}
\label{eq:super harm2}
v(x)\geq \negint_{B_r(x)}\biggl(v(y)-C_0\frac{|x-y|^2}{2n}\biggr)\,dy\\
\geq
\negint_{B_r(x)} v -C_0\frac{r^2}{2n}\qquad\text{whenever }B_r(x)\subset \Omega.
\end{multline}
To simplify the notation, set $C_1:=\frac{C_0}{2n}$. Then \eqref{eq:super harm2} can be rewritten as
\begin{equation}
\label{eq:super harm2 bis}
v(x)\geq 
\negint_{B_r(x)} v -C_1r^2\qquad \forall\,r \in (0,{\rm dist}(x,\partial\Omega)).
\end{equation}
On the other hand, since $\Delta v = -\Delta\varphi$ inside $\{v>0\}\cap \Omega$ (see \eqref{eq:v}),
we can apply \eqref{eq:sub harm}  to the function $w(y):=v(y)+C_0\frac{|x-y|^2}{2n}$ to get
\begin{equation}
\label{eq:harm2}
v(x)\leq 
\negint_{B_r(x)} v +C_1r^2\qquad \text{whenever }B_r(x)\subset \{v>0\}\cap \Omega.
\end{equation}
Now, given a point $\bar x \in \{v>0\}\cap \Omega$ such that ${\rm dist}(\bar x,\partial\{v>0\})\leq \frac13{\rm dist}(\bar x,\partial\Omega)$, we set
$$
{\bar r}:={\rm dist}(\bar x,\partial\{v>0\}),
\qquad {\bar y}:=\partial B_{{\bar r}}(\bar x)\cap \partial\{v>0\}.
$$
In this way $|{\bar y}-\bar x|={\bar r}$, and we can apply \eqref{eq:super harm2 bis} at ${\bar y}$ with $r=2{\bar r}$ to get
\begin{equation}
\label{eq:super harm3}
0=v({\bar y})\geq \negint_{B_{2{\bar r}}({\bar y})} v -4C_1\bar r^2.
\end{equation}
(note that, since ${\rm dist}(\bar x,\partial\{v>0\})\leq \frac13{\rm dist}(\bar x,\partial\Omega)$, $B_{2{\bar r}}({\bar y})\subset \Omega$).
On the other hand, it follows by \eqref{eq:harm2} applied at $\bar x$ with $r={\bar r}$ that
\begin{equation}
\label{eq:harm3}
v(\bar x)\leq 
\negint_{B_{{\bar r}}(\bar x)} v +C_1\bar r^2.
\end{equation}
Hence, since $v\geq 0$ and $B_{\bar x}(\bar x)\subset B_{2{\bar r}}(\bar y)$, combining \eqref{eq:super harm3} and \eqref{eq:harm3}
we get
\begin{multline*}
v(\bar x)-C_1\bar r^2\leq \negint_{B_{{\bar r}}(\bar x)}v =\frac{1}{|B_{\bar r}|} \int_{B_{{\bar r}}(\bar x)} v
\leq \frac{1}{|B_{\bar r}|} \int_{B_{2{\bar r}}(\bar y)} v\\
=\frac{|B_{2\bar r}|}{|B_{\bar r}|} \negint_{B_{2{\bar r}}({\bar y})} v =
2^n\negint_{B_{2{\bar r}}({\bar y})} v \leq 2^{n+2}C_1\bar r^2.
\end{multline*}
Recalling the definition of $\bar r$, this proves \eqref{eq:r2 above}.
\end{proof}

\begin{proof}[Proof of Theorem \ref{thm:C11}]
We begin by recalling the interior estimates for harmonic functions (see for instance \cite[Chapter 2.2.3, Theorem 7]{Ev}): if $\Delta w=0$ inside $B_r(x)$, then
\begin{equation}
\label{eq:elliptic}
r|\nabla w(x)|+ r^2|D^2w(x)|\leq C(n) \|w\|_{L^\infty(B_r(x))}.
\end{equation}
Let $x \in \{v>0\}\cap \Omega$,  assume that ${\rm dist}(x,\partial\{v>0\})\leq {\rm dist}(x,\partial\Omega),$
and set ${r}:={\rm dist}(x,\partial\{v>0\})$.
Consider the function
$$
w(y):=u(y)-\varphi(x)- \nabla\varphi(x)\cdot (y-x)
$$
and note that 
$\Delta w=0$  inside $B_{r}(x)$ (see Corollary \ref{cor:harm}).
Also, since $\varphi \in C^{1,1}(\Omega)$, 
\begin{equation}
\label{eq:w v}
\begin{split}
|w(y)-v(y)|&= |\varphi(y)-\varphi(x)-\nabla\varphi(x)\cdot (y-x)|\\
&\le\frac12 \int_0^1|D^2\varphi|(tx+(1-t)y)\,dt\,|y-x|^2\\
&\leq \frac12\|D^2\varphi\|_{L^\infty(\Omega)}r^2 \qquad \forall\,y \in B_{r}(x).
\end{split}
\end{equation}
Hence, applying \eqref{eq:elliptic} to 
$w$ inside $B_{r}(x)$,
 using \eqref{eq:r2 above} and \eqref{eq:w v} we obtain 
$$
\frac{|\nabla w(x)|}{r}+|D^2w(x)|\leq C(n)\Bigl(\frac12\|D^2\varphi\|_{L^\infty(\Omega)}+C'\Bigr),
$$
provided $r={\rm dist}(x,\partial\{v>0\})\leq \frac13{\rm dist}(x,\partial\Omega)$.
Noticing that $\nabla v(x)=\nabla w(x)$ and that $D^2v=D^2w-D^2\varphi$, this  yields
\begin{equation}
\label{eq:v hat C}
\frac{|\nabla v(x)|}{r}+|D^2v(x)|\leq C(n)\Bigl(\frac12\|D^2\varphi\|_{L^\infty(\Omega)}+C'\Bigr)+\|D^2\varphi\|_{L^\infty(\Omega)}=:\hat C.
\end{equation}
Hence, if we set
$$
\hat \Omega_{\tau}:=\biggl\{x \in \{v>0\}\cap\Omega\,:\,{\rm dist}(x,\partial\{v>0\})\leq \tau\,{\rm dist}(x,\partial\Omega)\biggr\},
$$
$$
\hat \Omega_{\tau}':=\biggl\{x \in \Omega\,:\,{\rm dist}(x,\{v=0\})\leq \tau\,{\rm dist}(x,\partial\Omega)\biggr\}=\hat\Omega_\tau \cup\bigl(\{v=0\}\cap \Omega\bigr),
$$
then \eqref{eq:v hat C} proves that the function
$$
V(x):=\frac{|\nabla v(x)|}{{\rm dist}(x,\partial \{v>0\})}+|D^2v(x)|
$$
 is uniformly bounded inside $\hat \Omega_{1/3}$
 by a constant $\hat C$.

To conclude the proof, we note that the boundedness of $V$  and \eqref{eq:r2 above} imply that 
$$
|v(x)|\leq C'{\rm dist}(x,\partial \{v>0\})^2\quad \text{and}\quad |\nabla v(x)|\leq \hat C\, {\rm dist}(x,\partial \{v>0\})\qquad \forall\, x\in \hat \Omega_{1/3},
$$
therefore
$$
v(x)\to 0 \quad \text{and}\quad \nabla v(x)\to 0\qquad \text{as ${\rm dist}(x,\partial \{v>0\})\to 0$}.
$$  
Thus, both $v$ and $\nabla v$ extend continuously to zero inside the set $\{v=0\}$, therefore $v \in C^1(\hat \Omega_{1/3}')$.
We now prove that $\nabla v$ is locally Lipschitz.

To this aim, we consider $x,y \in \hat\Omega_{1/6}'$. If $x,y \in \{v=0\}$  then $\nabla v(x)=\nabla v(y)=0$ and the result is trivially true. So we can assume that at least one of the two points belongs to $\{v>0\}$,
and that (with no loss of generality)
\begin{equation}
\label{eq:wlog dist}
{\rm dist}(x,\partial \{v>0\})\geq {\rm dist}(y,\partial \{v>0\}).
\end{equation}

We distinguish two cases.\\
{\it - Case 1:} $|x-y|\leq {\rm dist}(x,\partial \{v>0\})$. If we set $r:=|x-y|$, since $x \in \hat\Omega_{1/6}'$ it follows that $B_r(x)\subset \hat\Omega_{1/3}$. Thus, recalling that $|D^2v|\leq V\leq \hat C$ inside $\hat \Omega_{1/3}$, we can estimate
$$
|\nabla v(x)-\nabla v(y)| \leq \int_0^1|D^2v|(tx+(1-t)y)\,dt\,|x-y|\leq \hat C|x-y|.
$$
{\it - Case 2:} $|x-y|\geq {\rm dist}(x,\partial \{v>0\})$. Recalling \eqref{eq:wlog dist}, we can use again the boundedness of $V$ to deduce that
\begin{multline*}
\frac{|\nabla v(x)-\nabla v(y)|}{|x-y|}
\leq  \frac{|\nabla v(x)|+|\nabla v(y)|}{|x-y|}\\
\leq \frac{\hat C\bigl({\rm dist}(x,\partial \{v>0\})+ {\rm dist}(y,\partial \{v>0\})\bigr)}{|x-y|}\leq 2\hat C.
\end{multline*}
This proves that $\nabla v \in C^{0,1}(\hat\Omega_{1/6}')$, thus $v \in C^{1,1}(\hat\Omega_{1/6}')$.
On the other hand, because  $\Delta u=0$ inside $\{v>0\}$ and $\varphi \in C^{1,1}(\Omega)$, $v \in C^{1,1}_{\rm loc}(\Omega\setminus \{v=0\})$.

Since $\hat\Omega_{1/6}'\cup (\Omega\setminus \{v=0\})=\Omega$, the result follows.
\end{proof}

\section{A new formulation of the obstacle problem}
In the previous section we proved that $v=u-\varphi \in C^{1,1}_{\rm loc}(\Omega)$ provided the obstacle $\varphi$ is of class $C^{1,1}$.
For now on we will always make this assumption.

We now observe that, for a $C^{1,1}$ function $w$, it holds\footnote{To prove this, one can apply \cite[Corollary 1, Chapter 3.1.2]{EG}
to the Lipschitz functions $w$ and $\nabla w$
to deduce that
$$
D^2 w=0\quad \text{a.e. inside }\{\nabla w=0\}\qquad \text{and}\qquad 
\nabla w=0\quad \text{a.e. inside }\{w=0\}.
$$
Hence
$$
D^2 w=0\qquad \text{a.e. inside }\{\nabla w=0\}\cap\{w=0\}
\qquad \text{and}\qquad
|\{w=0\}\setminus\{\nabla w=0\}|=0,
$$
which implies the result.
}
$$
D^2w=0\qquad \text{a.e. inside }\{w=0\}.
$$
Applying this result
to $v$ we deduce that $\Delta v={\rm tr}(D^2v)=0$ a.e. inside $\{v=0\}$, so recalling \eqref{eq:v} we obtain that $v$ solves
\begin{equation}
\label{eq:obst var}
\Delta v=(-\Delta \varphi)\chi_{\{v>0\}}
\qquad \text{a.e. inside }\Omega.
\end{equation}
Note that, since $\Delta v \leq -\Delta \varphi$ (see \eqref{eq:v}), we also have
$$
0 \leq -\Delta\varphi \qquad \text{on }\{v=0\}.
$$
In particular, if we assume that $-\Delta \varphi>0$ then it follows by \eqref{eq:obst var} that $\Delta v$ is discontinuous across the free boundary $\partial \{v>0\}$, which implies that $v$ cannot be $C^2$ even if $\varphi$ is smooth.
This shows that, in general, the $C^{1,1}$ regularity of $v$ is optimal.
\smallskip

From now on our goal is to study the regularity of the free boundary.
We note that, for this purpose, it is necessary to make some extra assumptions on $\varphi$.
To understand this, one can make the following observation: fix a ball $B_r(x)\subset\subset \{u>\varphi\}$, consider $K\subset B_r(x)$ a compact set, and let $\psi \in C^\infty(B_r(x))$ be a nonnegative function such that $0 \leq \psi \leq 1$, $K=\{\psi=1\}$, and $D^2\psi|_K=0$.
Then, consider the function
$$
\phi:=\varphi(1-\psi)+u\psi
$$
(see Figure \ref{Pic28}).
 \begin{figure}[ht]
\includegraphics[scale=0.18]{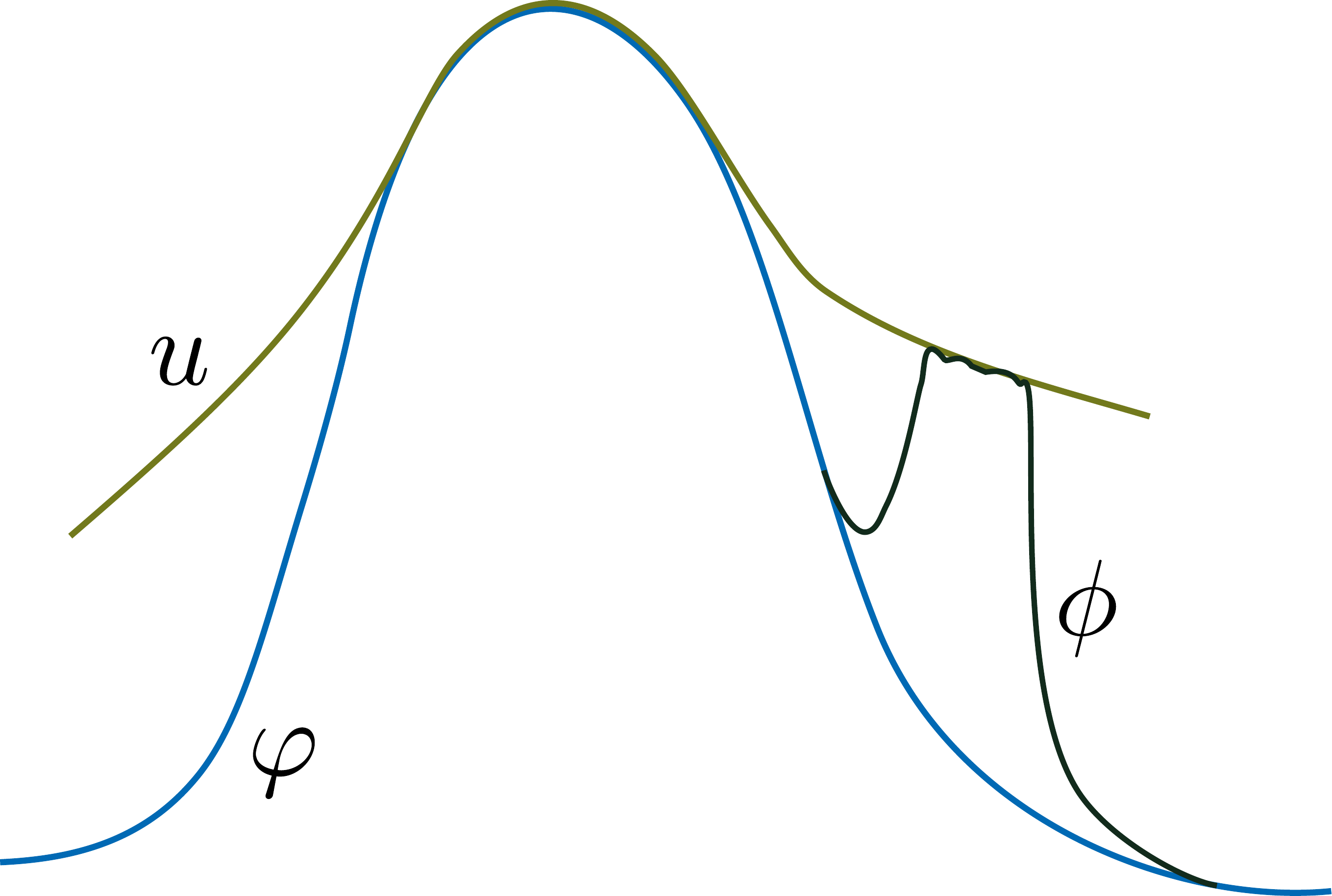} \caption{One can modify the obstacle away from the contact set so that the new contact set is very irregular.} 
  \label{Pic28}
 \end{figure}

 Note that, since $u$ is harmonic in the region $\{u>\varphi\}$, the function $\psi$ is as smooth as $\varphi$ (up to $C^\infty$ regularity).
Also, since $\varphi \leq \phi \leq u$, it holds\footnote{Recall the notation
$$
\mathcal K_\phi:=\{v \in W^{1,2}(\Omega)\,:\,v|_{\partial\Omega}=f,\,v \geq \phi\},\qquad
\mathcal K_\varphi:=\{v \in W^{1,2}(\Omega)\,:\,v|_{\partial\Omega}=f,\,v \geq \varphi\}.
$$}
$$
u \in \mathcal K_\phi\subset \mathcal K_\varphi.
$$
Hence, since $u$ minimizes the Dirichlet integral in $\mathcal K_\varphi$, so it does in $\mathcal K_\phi$,
proving that $u$ solves the obstacle problem with obstacle $\phi$. 

We now notice that 
$$
\{u=\phi\}=\{u=\varphi\}\cup K,
$$
where $K$ was an arbitrary compact set, showing that in general one cannot hope to prove any regularity of the free boundary.
However, if one analyses the example above, we note that on the set $K$ the obstacle $\phi$ has Laplacian $0$.
Hence, this example would be ruled out if we assumed that the obstacle has negative Laplacian, at least in a neighborhood of the free boundary.

For this reason,  we shall assume that $\varphi$ is smooth and that $\Delta\varphi<0.$
Actually, as noticed in \cite{C98,W99,M03,PSU12}, from the point of view of the local structure of the free boundary it suffices to understand the model case $\Delta\varphi\equiv -1$.
Thus, from now on we shall focus on the equation
$$
\Delta v=\chi_{\{v>0\}},\qquad v \geq 0.
$$
Noticing that $v|_{\partial\Omega}=u|_{\partial\Omega}-\varphi=f-\varphi$,
up to replacing $v$ by $u$ and $f-\varphi$ by $f$, our problem becomes to investigate the regularity of $\partial\{u>0\}$ 
for a function $u \in C^{1,1}_{\rm loc}(\Omega)$
 satisfying
\begin{equation}
\label{eq:obst}
\left\{
\begin{array}{ll}\Delta u=\chi_{\{u>0\}}&\text{in }\Omega,\\
u \geq 0&\text{in }\Omega,\\
u=f &\text{on }\partial\Omega,
\end{array}
\right.
\end{equation}
see Figure \ref{Pic-ell-ob}.

 \begin{figure}[ht]
\includegraphics[scale=0.25]{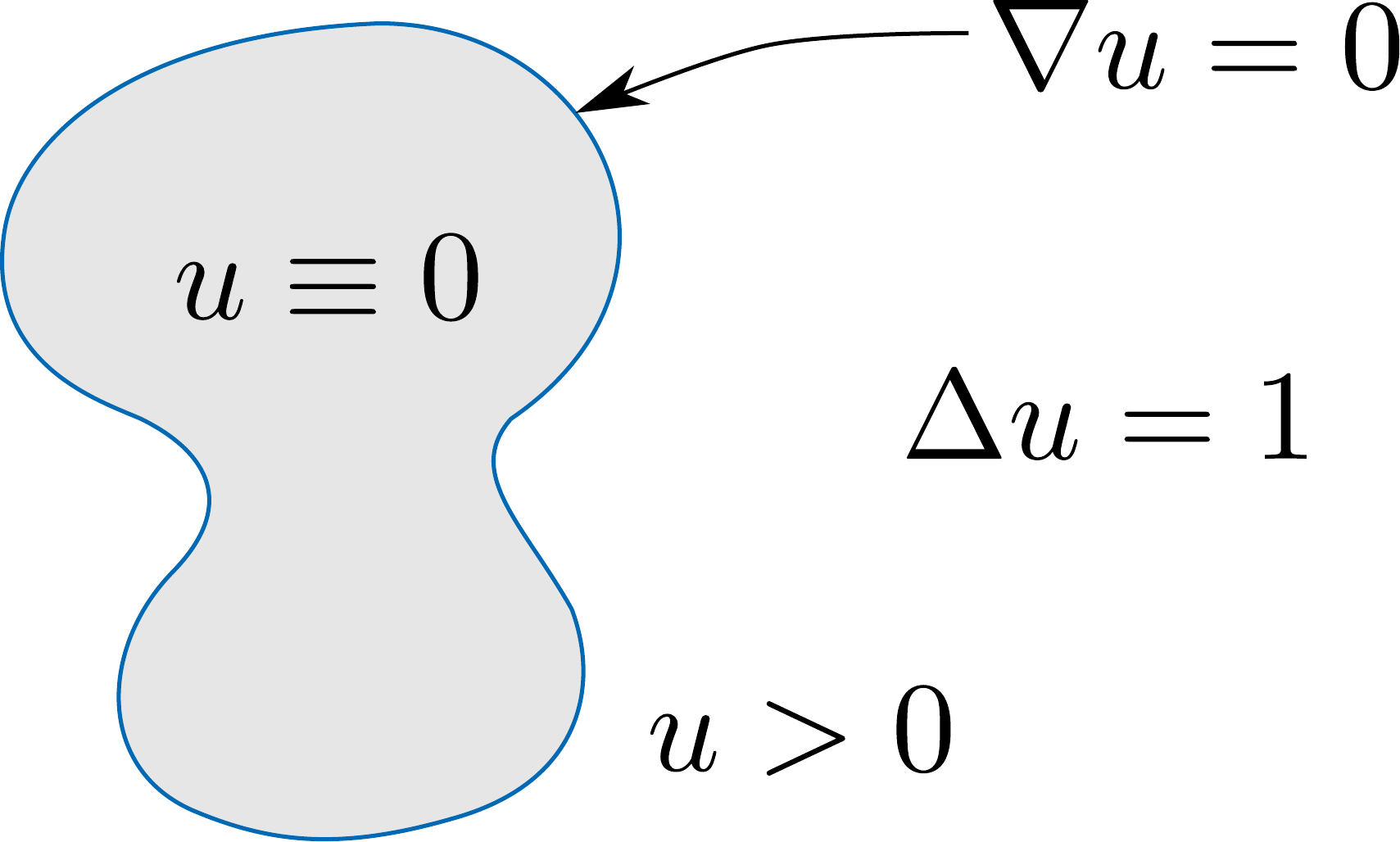} \caption{A new formulation of the obstacle problem.} 
  \label{Pic-ell-ob}
 \end{figure}

\section{Non-degeneracy of solutions}

As shown in Lemma \ref{lem:quadr},  $u$ grows at most quadratically away from the free boundary. Here
we state (and prove) the result in a slightly different form:
\begin{lemma}
\label{cor:C11}
Let $u$ solve \eqref{eq:obst},
let $\Omega'\subset \subset \Omega$,
let $x_0 \in \partial\{u>0\}$, and assume that $B_{r}(x_0)\subset \Omega'$.
Then there exists $C=C(\Omega')$ such that
$$
\sup_{B_r(x_0)}u\leq C\,r^2.
$$
\end{lemma}
\begin{proof}
Given $x \in B_r(x_0)$, we can write $u(x)$ using the Taylor formula centered at $x_0:$
\begin{multline*}
u(x)=u(x_0)+\nabla u(x_0)\cdot (x-x_0)
+\int_0^1(1-t)D^2u\bigl(x_0+t(x-x_0)\bigr)[x-x_0,x-x_0]\,dt.
\end{multline*}
Since $u(x_0)=0$ and $\nabla u(x_0)=0$,
setting  $C:=\|D^2u\|_{L^\infty(\Omega')}$ 
we get
$$
0\leq u(x)\leq  \frac{C}{2}|x-x_0|^2\leq \frac{C}{2}r^2,
$$
as desired.
\end{proof}

As shown in \cite{C77,C98}, 
the upper bound is optimal.
\begin{proposition}
\label{prop:non deg}
Let $u$ solve \eqref{eq:obst},
let $x_0 \in \partial\{u>0\}$, and assume that $B_{r}(x_0)\subset \Omega$.
Then there exists  a dimensional constant $c=c(n)>0$ such that
$$
\sup_{B_r(x_0)}u\geq c\,r^2.
$$
\end{proposition}
\begin{proof}
Assume first that $x_0 \in \{u>0\}$, and set
$$
h_{x_0}(x):=u(x)-\frac{|x-x_0|^2}{2n}.
$$
Note that 
$$
\Delta h_{x_0}=\Delta u -1 =0 \qquad \text{ inside }B_r(x_0)\cap \{u>0\},
$$
and that $h_{x_0}(x_0)=u(x_0)>0$.
Hence, by the maximum principle for harmonic functions,
\begin{equation}
\label{eq:max}
0<h_{x_0}(x_0)\leq \sup_{B_r(x_0)\cap \{u>0\}}h_{x_0}=\max_{\partial \left(B_r(x_0)\cap \{u>0\}\right)}h_{x_0}.
\end{equation}
Since $\partial \left(B_r(x_0)\cap \{u>0\}\right)=\left(\partial B_r(x_0)\cap \{u>0\}\right)\cup \left(B_r(x_0)\cap \partial\{u>0\}\right)$
and
$$
h_{x_0}\leq u=0 \qquad \text{on }B_r(x_0)\cap \partial\{u>0\}
$$
it follows that the maximum in \eqref{eq:max} is attained on $\partial B_r(x_0)\cap \{u>0\}$, that is
$$
0<\max_{\partial B_r(x_0)\cap \{u>0\}}h_{x_0}.
$$
Since
$$
h_{x_0}=u-\frac{r^2}{2n}\qquad \text{on }\partial B_r(x_0)\cap \{u>0\},
$$
this yields
$$
\frac{r^2}{2n}<\max_{\partial B_r(x_0)\cap \{u>0\}}u\leq \sup_{B_r(x_0)}u,
$$
proving the result whenever 
$x_0 \in \{u>0\}$.

Now, if $x_0 \in \partial\{u>0\}$, consider a sequence of point $\{x_k\}_{k\geq 1} \subset \{u>0\}$ such that $x_k \to x_0$, and set $r_k:=r-|x_k-x_0|$. Since
$B_{r_k}(x_k)\subset B_r(x_0)\subset \Omega$ we get
$$
\sup_{B_{r_k}(x_k)}u\geq \frac{r_k^2}{2n},
$$
and the result follows letting $k\to \infty$.
\end{proof}

\section{Blow-up analysis and Caffarelli's dichotomy}
Thanks to Corollary \ref{cor:C11} and Proposition \ref{prop:non deg}, we know that
$$
\sup_{B_r(x_0)}u\simeq r^2 \qquad \forall\, x_0\in \partial\{u>0\}\text{ with }B_r(x_0)\subset \Omega.
$$
This suggests the following rescaling:
for $x_0\in \partial\{u>0\}$ and $r>0$ small, we define the family of functions
\begin{equation}
\label{eq:blow family}
u_{x_0,r}(x):=\frac{u(x_0+rx)}{r^2}.
\end{equation}
In this way, since $u \in C^{1,1}_{\rm loc}(\Omega)$, if we take $\rho<{\rm dist}(x_0,\partial\Omega)$ we get (see Figure \ref{Pic06-7}):
\begin{enumerate}
\item[$\bullet$] $u_{x_0,r}(0)=0,\qquad \sup_{B_1}u_{x_0,r}\simeq  1\quad \forall\,r \in (0,\rho);
$
\item[$\bullet$]
$
|D^2u_{x_0,r}|(x)=|D^2u|(x_0+rx)\leq C\qquad \forall\,x \in B_{r^{-1}\rho}(0).
$
\end{enumerate}
 \begin{figure}[ht]
\includegraphics[scale=0.25]{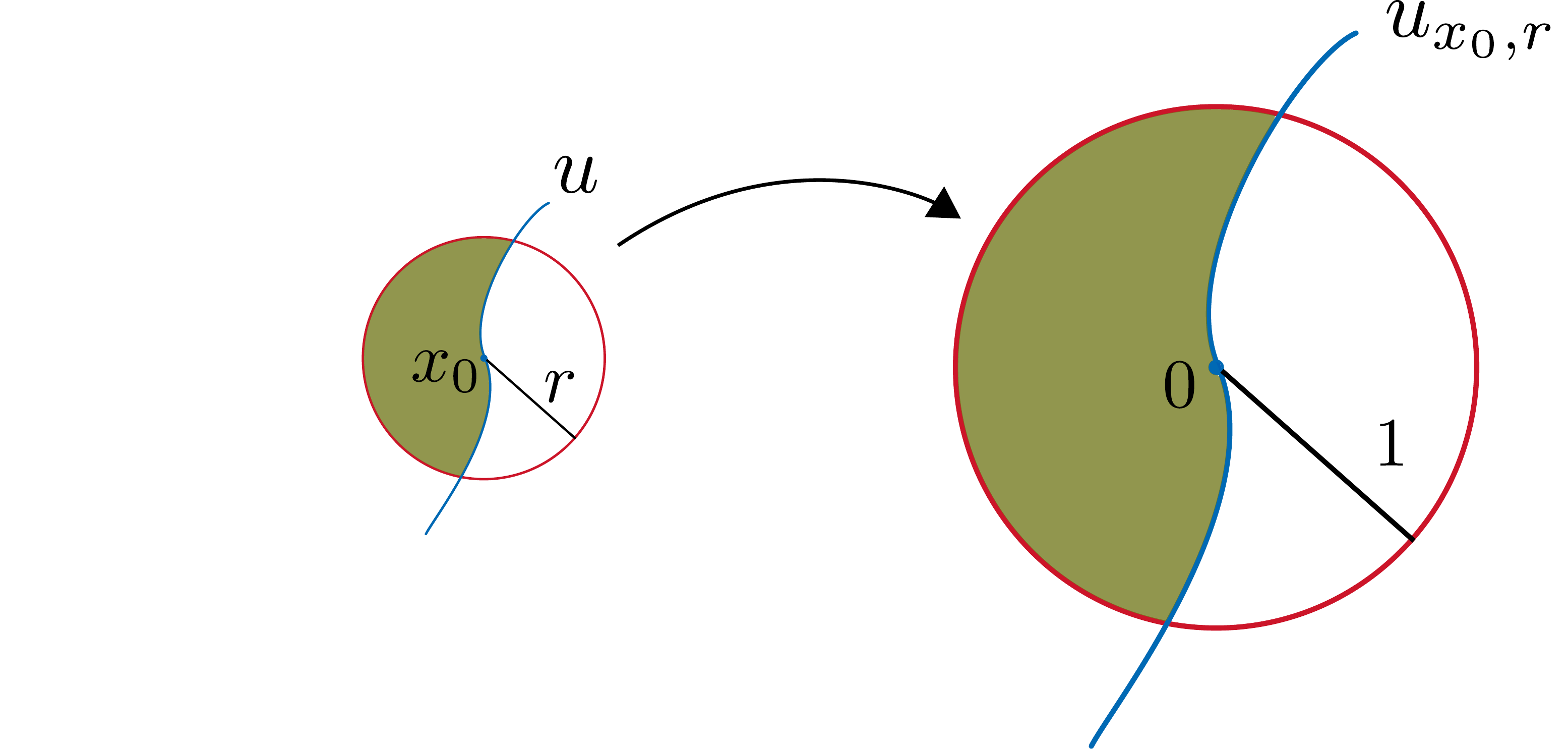}
\caption{By scaling, we look at functions of size $1$ defined inside $B_1$.} 
\label{Pic06-7}
\end{figure}
Note that, by Ascoli-Arzel\`a Theorem, the family of functions  $\{u_{x_0,r}\}_{r>0}$
are compact in $C^1_{\rm loc}$. So, one can consider a possible limit (up to subsequences) as $r\to 0^+.$ Such a limit is usually called a ``blow-up''.
Observe that, since the function $u_{x_0,r}$ is defined in the ball $B_{r^{-1}\rho}(0)$, any blow-up is defined in the whole $\R^n$.

The first goal is to classify the possible blow-ups, since they give us an idea of the infinitesimal behavior of $u$ near $x_0$.
We begin by considering two possible natural type of blow-ups that one may find.

\subsection{Regular free boundary points}
Let us first imagine that the free boundary is smooth near $x_0$, with $u>0$ on one side and $u=0$ on the other side.
In this case, as we rescale $u$ around $x_0$, we expect in the limit to see a one dimensional ``half-parabola'', see Figure \ref{Fig-blow1}.

\begin{figure}[ht]
\includegraphics[width=0.26\textwidth]{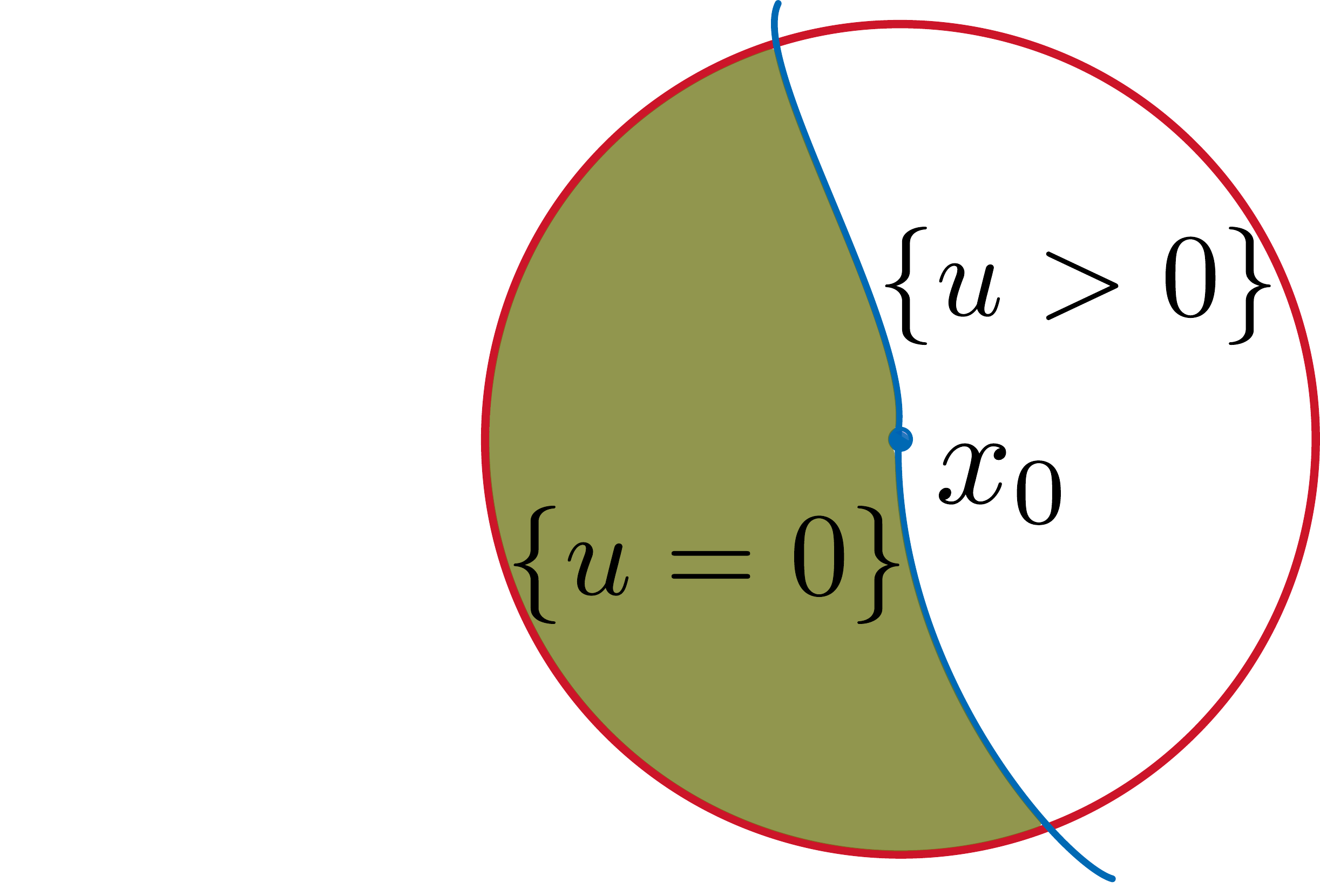}
\hspace{1cm}
\includegraphics[width=0.22\textwidth]{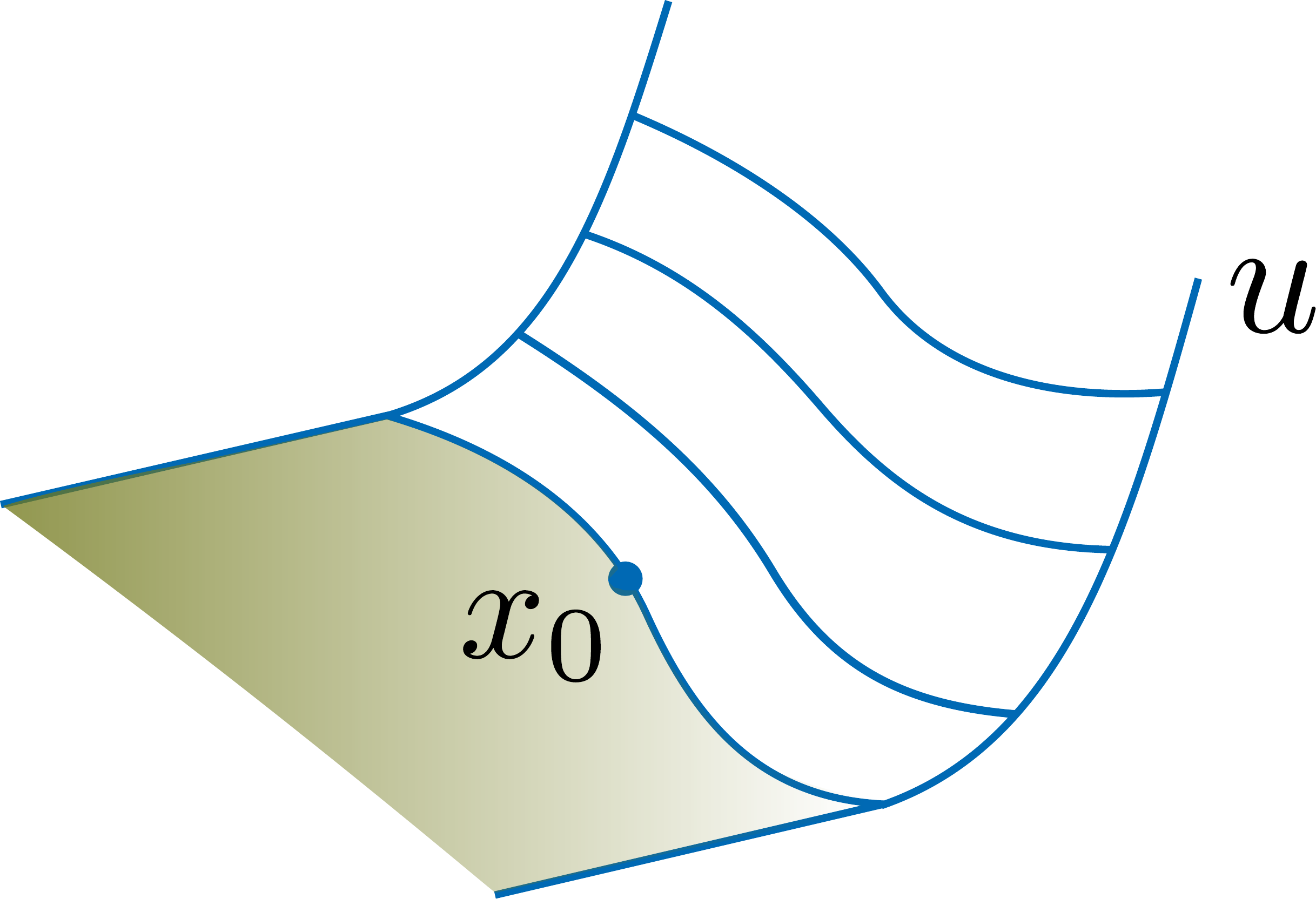}
\hspace{0.8cm}
\includegraphics[width=0.23\textwidth]{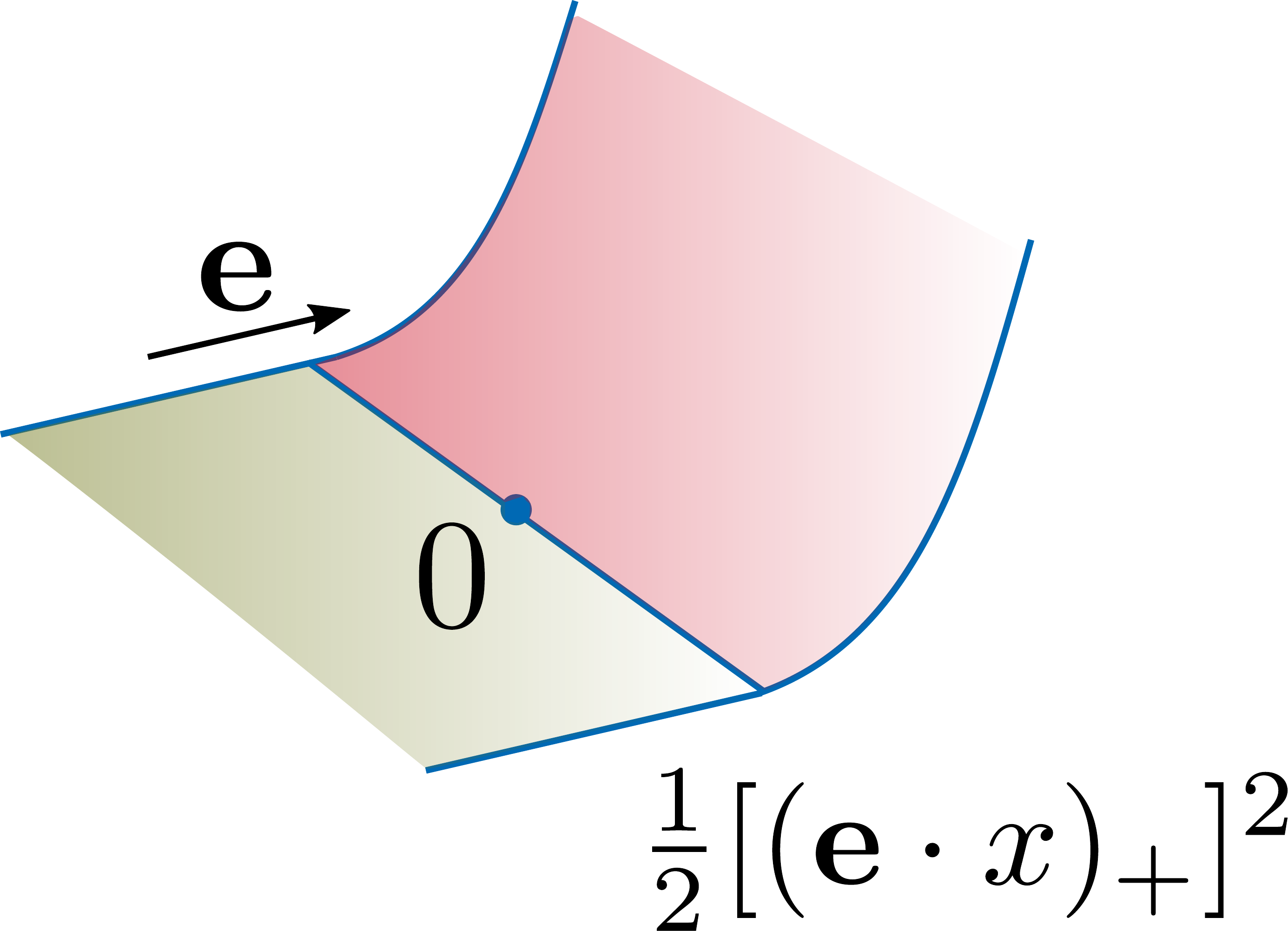}
\caption{Performing a blow-up near a ``thick'' free boundary points.}
\label{Fig-blow1}
\end{figure}

This motivates the following:
\begin{definition}
\label{def:reg pt}
A free boundary point $x_0\in \partial\{u>0\}$ is called a {\it regular point} if, up to a subsequence of radii,
$$
\frac{u(x_0+rx)}{r^2}\to \frac12 [(\mathbf{e}\cdot x)_+]^2\qquad \text{as }r\to 0^+
$$
for some unit vector $\mathbf{e}\in \mathbb S^{n-1}$.
\end{definition}

\subsection{Singular free boundary points}
Suppose now that the contact set is very narrow near $x_0$. Since $\Delta u=1$ outside of the contact set,
as
we rescale $u$ around $x_0$ we expect to see in the limit a function that has Laplacian equal to $1$ almost everywhere.
In dimension two, a natural behavior that one may expect to observe is represented in Figure \ref{Fig-blow2}.

\begin{figure}[ht]
\includegraphics[width=0.28\textwidth]{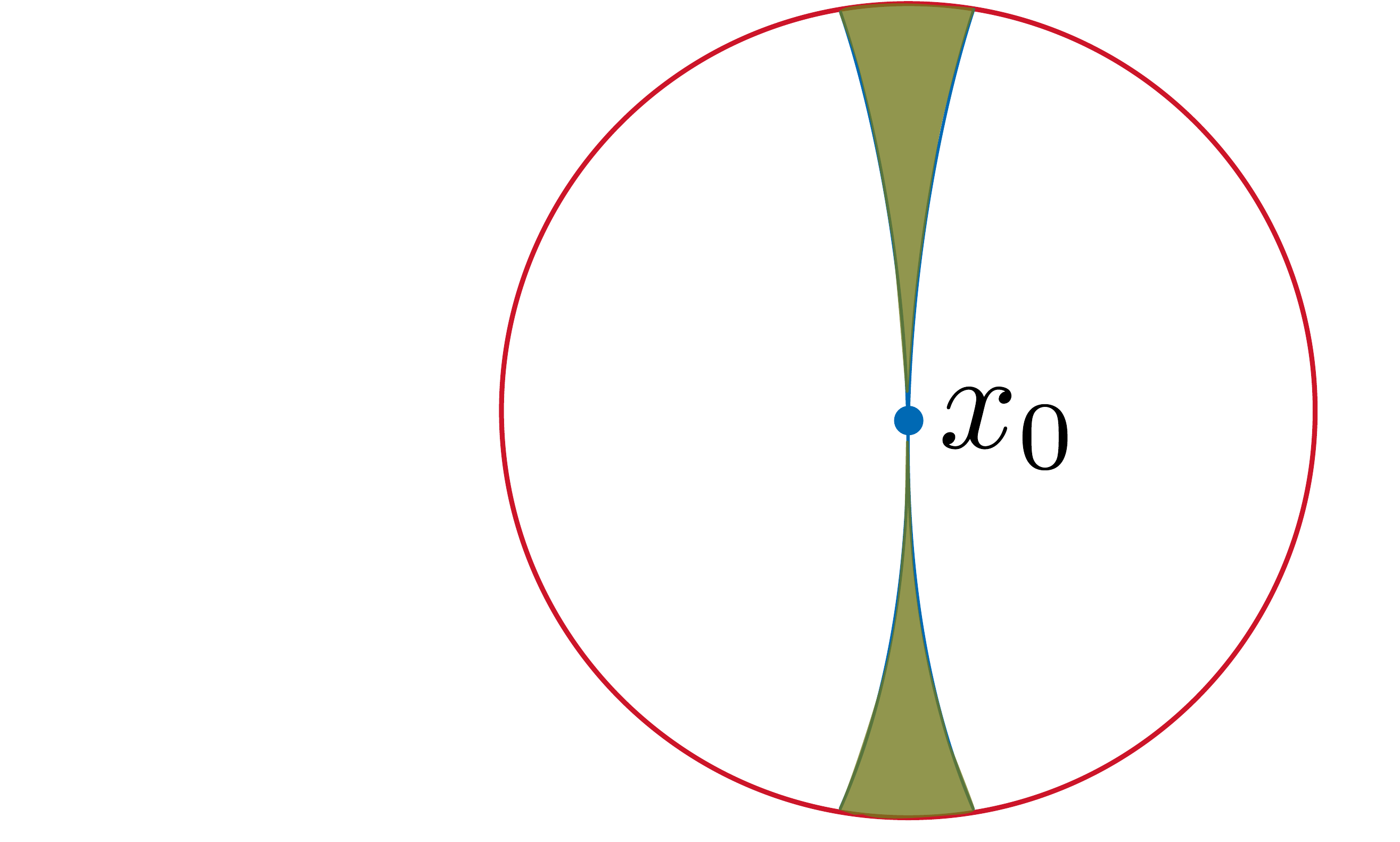}
\hspace{-0.2cm}
\includegraphics[width=0.28\textwidth]{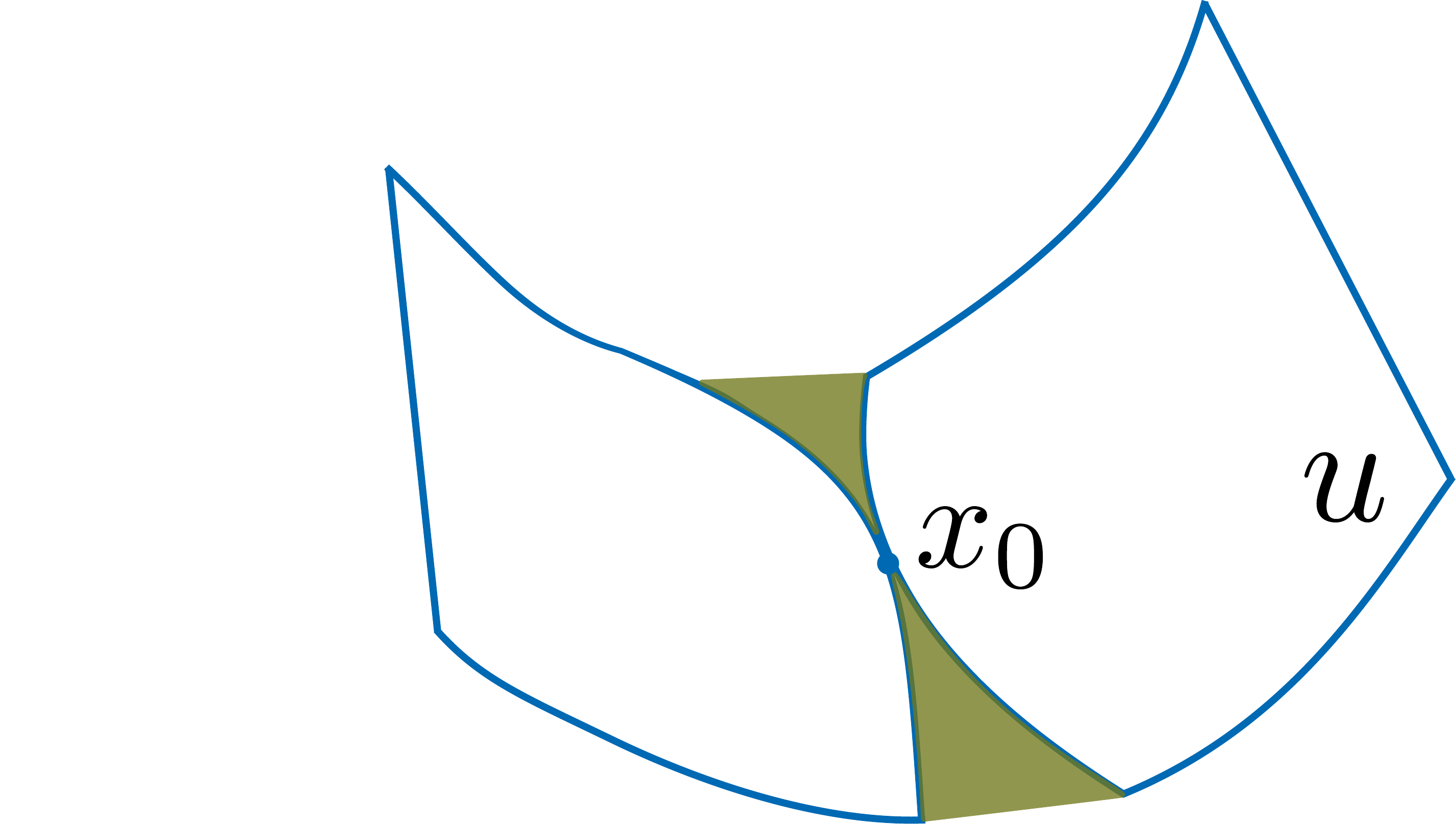}
\hspace{1cm}
\includegraphics[width=0.25\textwidth]{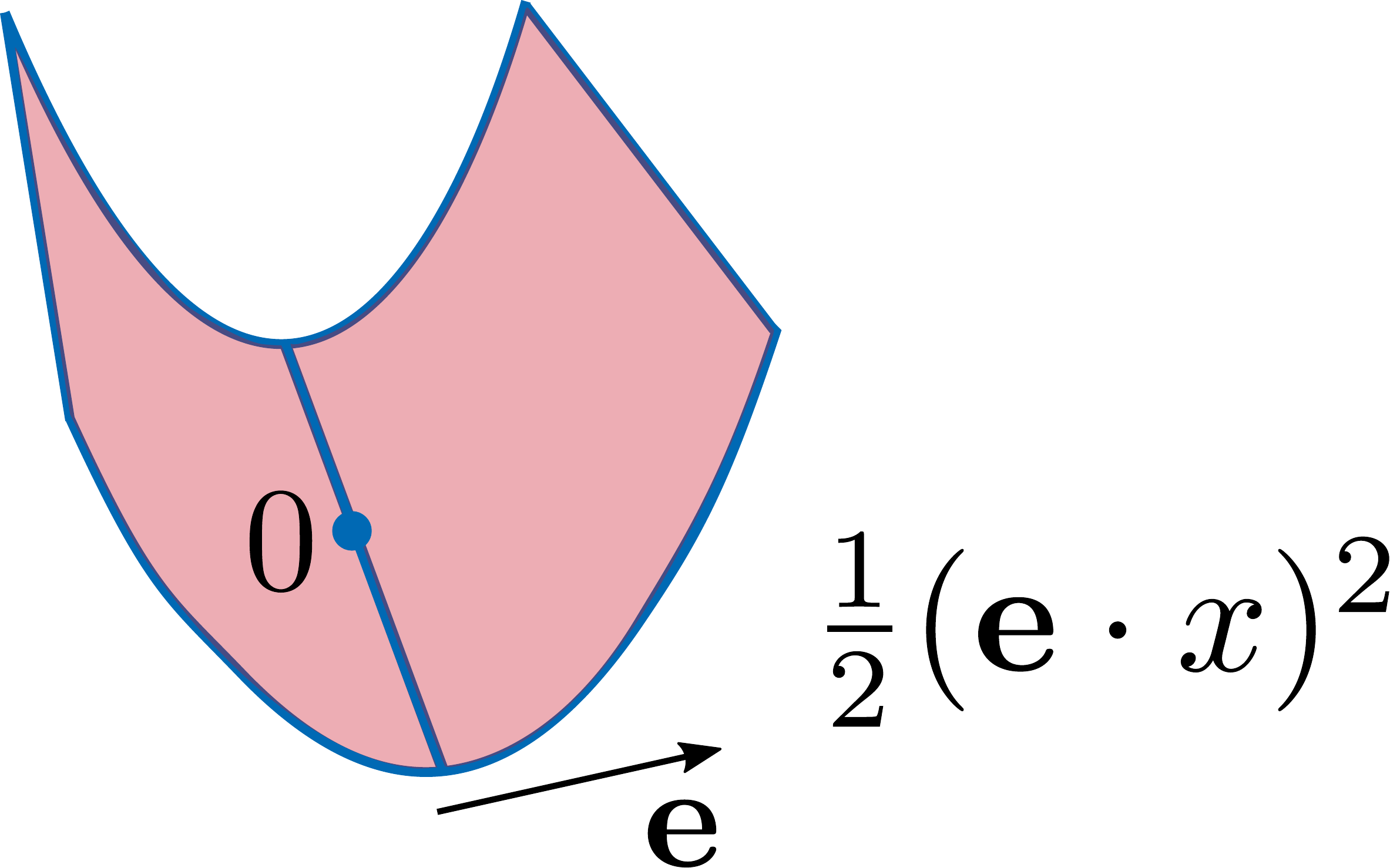}
\caption{Performing a blow-up near a ``thin'' free boundary points.}
\label{Fig-blow2}
\end{figure}

More in general, since any nonnegative quadratic polynomial with Laplacian 1 solves \eqref{eq:obst},
one introduces the following:
\begin{definition}
\label{def:sing pt}
A free boundary point $x_0\in \partial\{u>0\}$ is called a {\it singular point} if, up to a subsequence of radii,
$$
\frac{u(x_0+rx)}{r^2}\to p(x):=\frac12\langle Ax,x\rangle \qquad \text{as }r\to 0^+
$$
for some nonnegative definite matrix $A \in \R^{n\times n}$ with ${\rm tr}(A)=1$.
\end{definition} 
Note the form of the polynomial $p$ is strictly related to the shape of the contact set near $0$.
For instance, if $n=3$ and 
$p(x)=\frac{1}2(\mathbf{e}\cdot x)^2$ for some unit  vector $\mathbf{e}\in \mathbb S^{2}$, then the contact set is close to the 2-dimensional plane $\{\mathbf{e}\cdot x=0\}$, see Figure \ref{Fig-blow2bis}.
\begin{figure}[ht]
\includegraphics[width=0.25\textwidth]{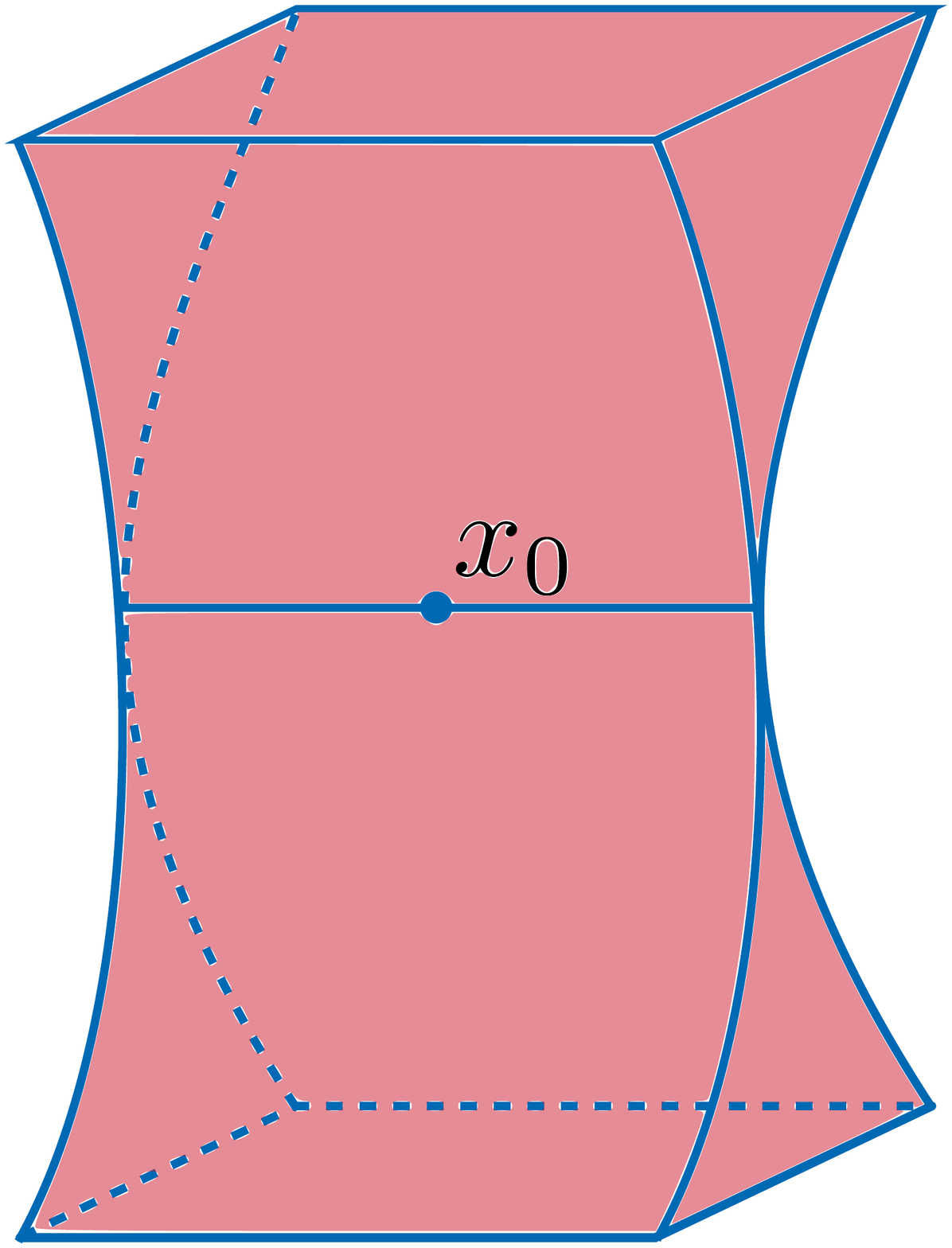}
\caption{A singular point in dimension 3 where the contact set is close to a plane.}
\label{Fig-blow2bis}
\end{figure}

 On the other hand, one may also expect to see points where the contact set is close to a line, that could correspond for instance to a polynomial of the form $p(x)=\frac{1}4(x_1^2+x_2^2)$, see Figure~\ref{Fig-blow2bis3}.
\begin{figure}[ht]
\includegraphics[width=0.1\textwidth]{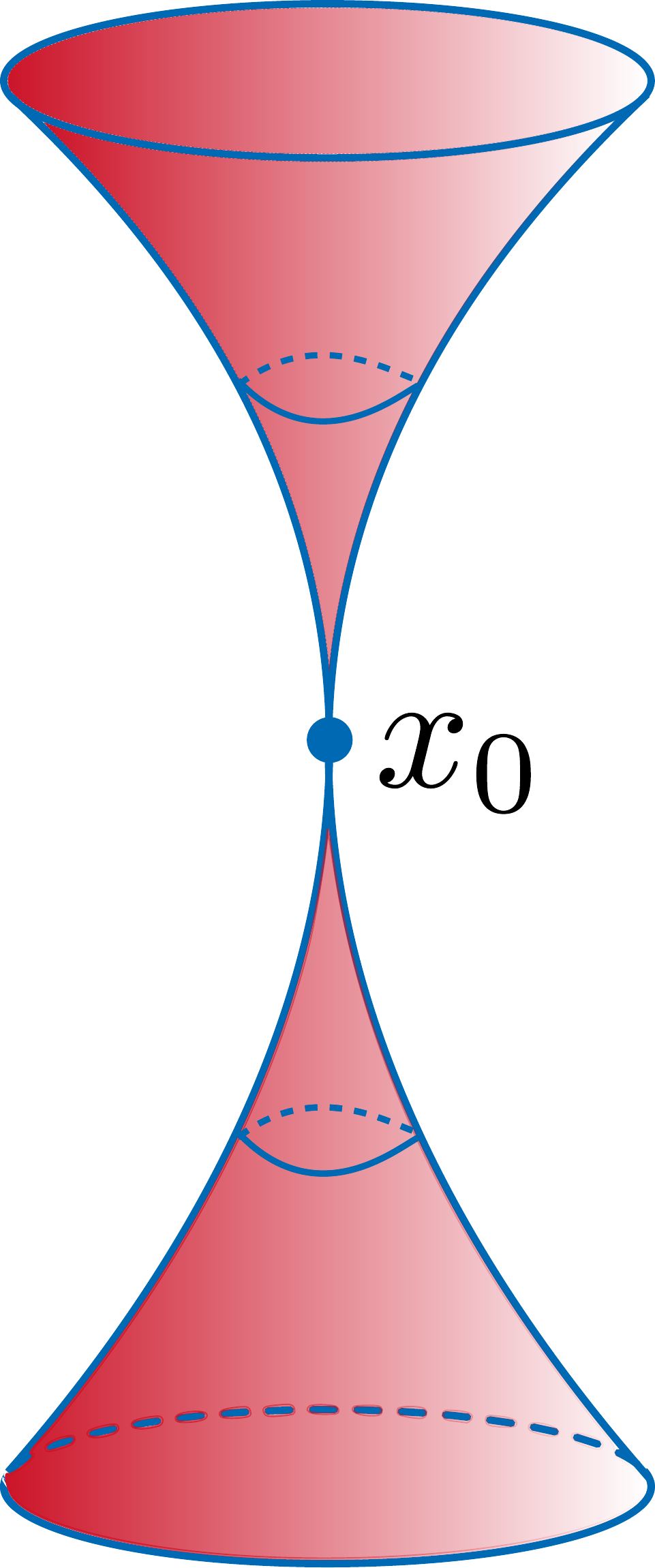}
\caption{A singular point in dimension 3 where the contact set is close to a line.}
\label{Fig-blow2bis3}
\end{figure}

\subsection{Caffarelli's dichotomy theorem}
Note that a priori the definitions of regular and singular points may not be mutually exclusive. Indeed,  a free boundary point could potentially be regular along some sequence of radii and singular along a different sequence. Also, it is not clear that regular and singular points should exhaust the whole free boundary.

These highly nontrivial and deep issues have been answered by Caffarelli in \cite{C77}:

\begin{theorem}
\label{thm:dico}
Let $u$ solve \eqref{eq:obst},
and let $x_0 \in \partial\{u>0\}\cap \Omega$. 
Then one of these two alternatives hold (see Figure \ref{Fig-reg-sing}):
\begin{enumerate}
\item[(i)]
either $x_0$ is regular, and there exists a radius $r_0>0$ such that $\partial\{u>0\}\cap B_{r_0}(x_0)$ is an analytic hypersurface consisting only of regular points;
\item[(ii)]
or $x_0$ is singular, in which case for any $r>0$ small there exists a unit vector $\mathbf e_r \in \mathbb S^{n-1}$ such that $\partial\{u>0\}\cap B_{r}(x_0)\subset \{x\,:\,|\mathbf{e}_r\cdot (x-x_0)|\leq o(r)\}$.
\end{enumerate}
\end{theorem}

\begin{figure}[ht]
\includegraphics[width=0.27\textwidth]{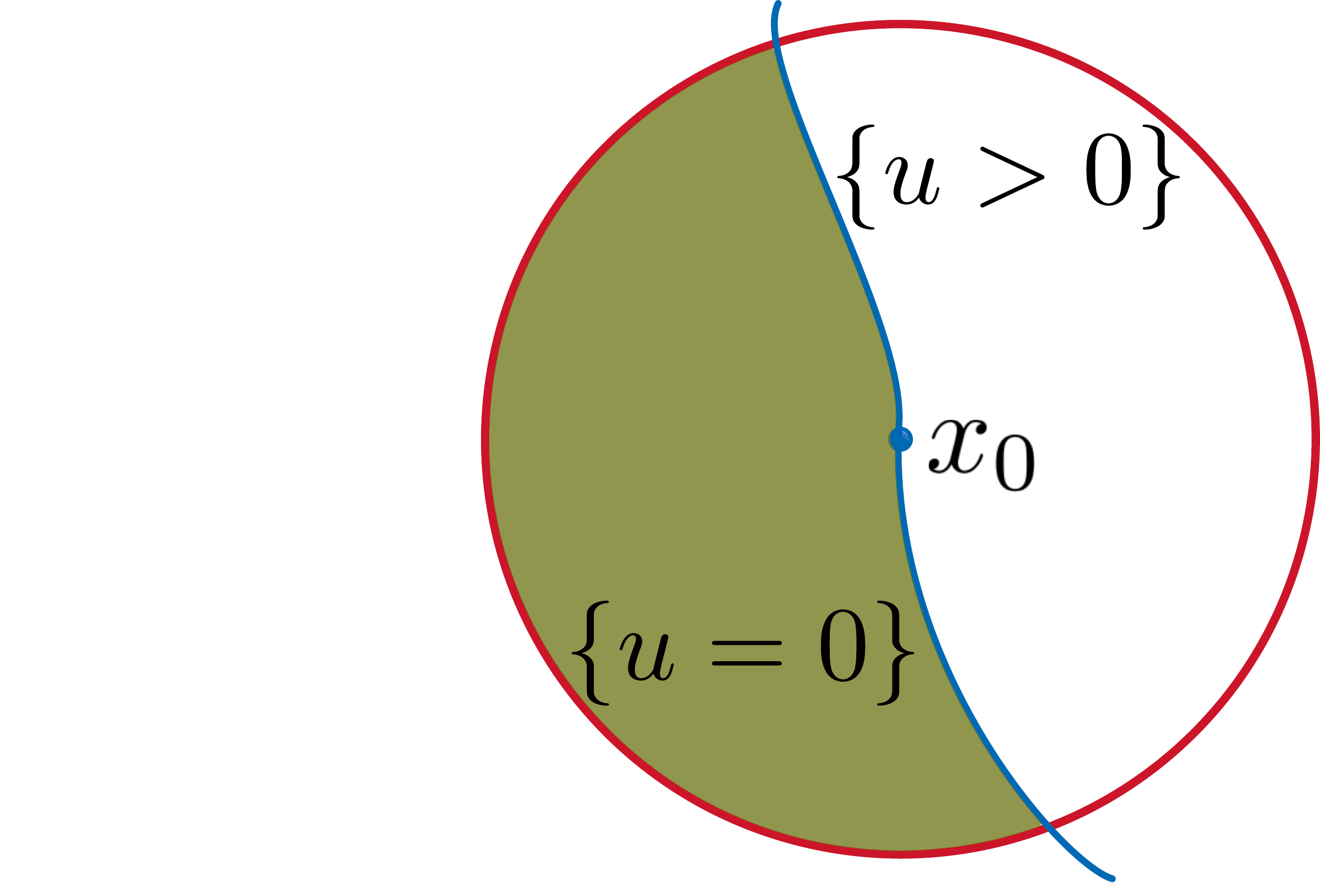}
\hspace{1.2cm}
\includegraphics[width=0.28\textwidth]{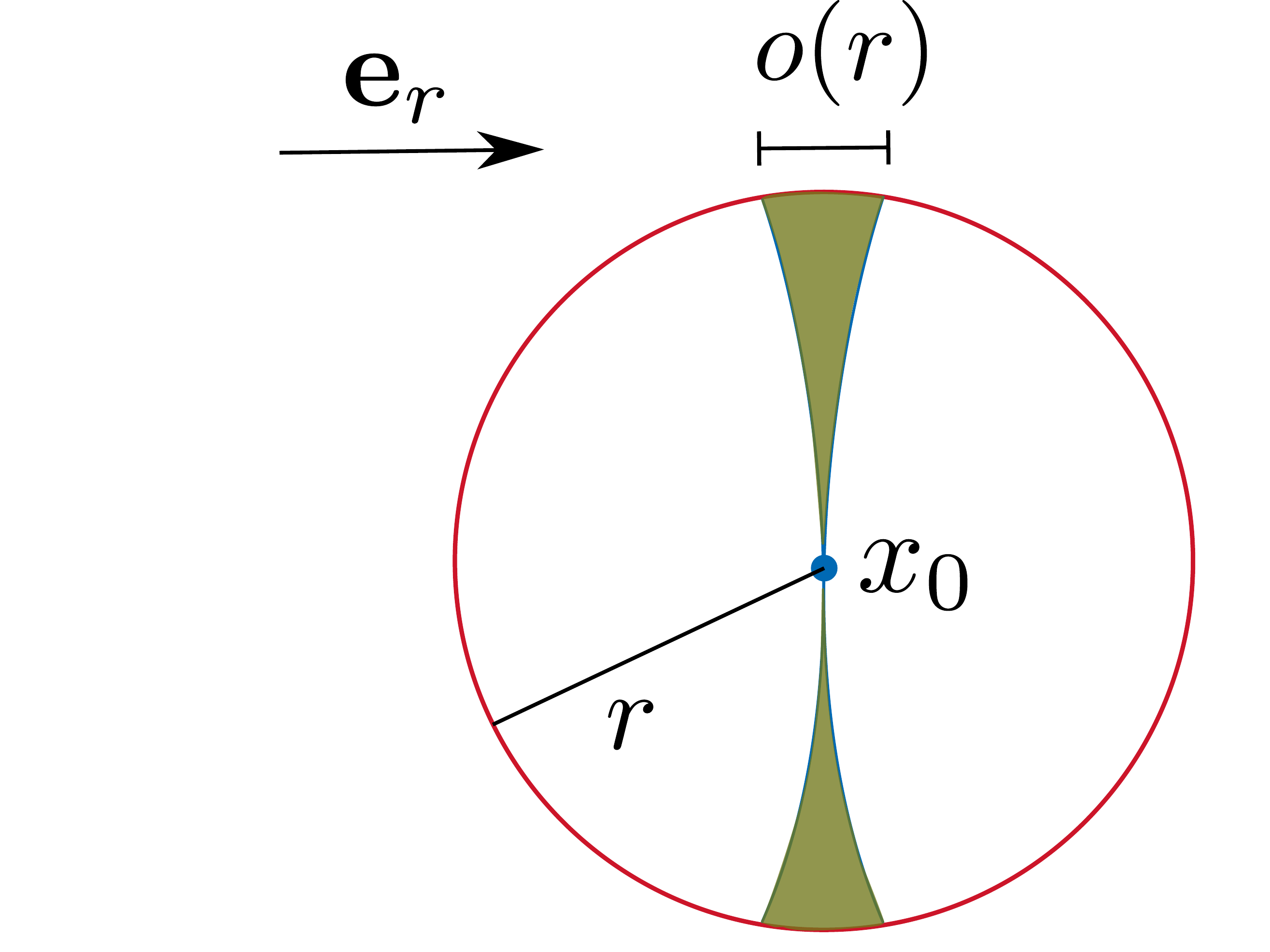}
\caption{A regular (left) and a singular (right) free boundary point.}
\label{Fig-reg-sing}
\end{figure}

\begin{proof}[Idea of the proof]
The first key ingredient is a semiconvexity estimate of the following form:
given $x \in \{u>0\}\cap \Omega$ set $r_{x}:={\rm dist}(x,\partial\{u>0\})$.
Then, for any unit vector $\mathbf e\in \mathbb S^{n-1}$,
$$ 
\partial_{\mathbf e\mathbf e}u(x)\geq -\omega(r_x),
$$
where $\omega:\R^+\to\R^+$ is a continuous increasing function such that $\omega(0)=0$ (see for instance \cite[Theorem 3]{C98} for a proof).

This fundamental bound allows one to show that blow-ups are convex.
Indeed, if $x_0 \in \partial\{u>0\}$
and $B_\rho(x_0)\subset \Omega,$
we can define $u_{x_0,r}:B_{r^{-1}\rho}(0)\to \R$ as in \eqref{eq:blow family}.
Then
$$
\partial_{\mathbf e\mathbf e}u_{x_0,r}(x)=\partial_{\mathbf e\mathbf e}u(x_0+rx)\geq -\omega(rx)\qquad \forall\,x\in B_{r^{-1}\rho}(0),\,\forall\,r>0.
$$
In particular, if we let $r\to 0^+$ and  $u_{x_0}:\R^n\to \R$ denotes a possible limit point, then
$$
\partial_{\mathbf e\mathbf e}u_{x_0}(x)\geq0 \qquad \forall\,x\in \R^n,
$$
hence $u_{x_0}\geq 0$ is convex.
One now distinguishes between two cases, depending on the properties of the contact set $\{u_{x_0}=0\}$. More precisely:

\smallskip

\noindent
$\bullet$ {\it Case 1: The contact set $\{u_{x_0}=0\}$ has positive measure.} In this case, since $0 \in \partial\{u_{x_0}>0\}$ and the set $\{u_{x_0}=0\}$ is convex, it must contain a ball $B_\sigma(\hat x)$ disjoint from the origin (see Figure \ref{Fig-contact pos}).
Consider set of directions $\mathbf w\in \mathbb S^{n-1}$ of the form $\mathbf w=-\frac{\mathbf v}{|\mathbf v|}$ with $\mathbf v \in B_{\sigma/2}(\hat x)$.
\begin{figure}[ht]
\hspace*{-2.2cm}\includegraphics[width=1.0\textwidth]
{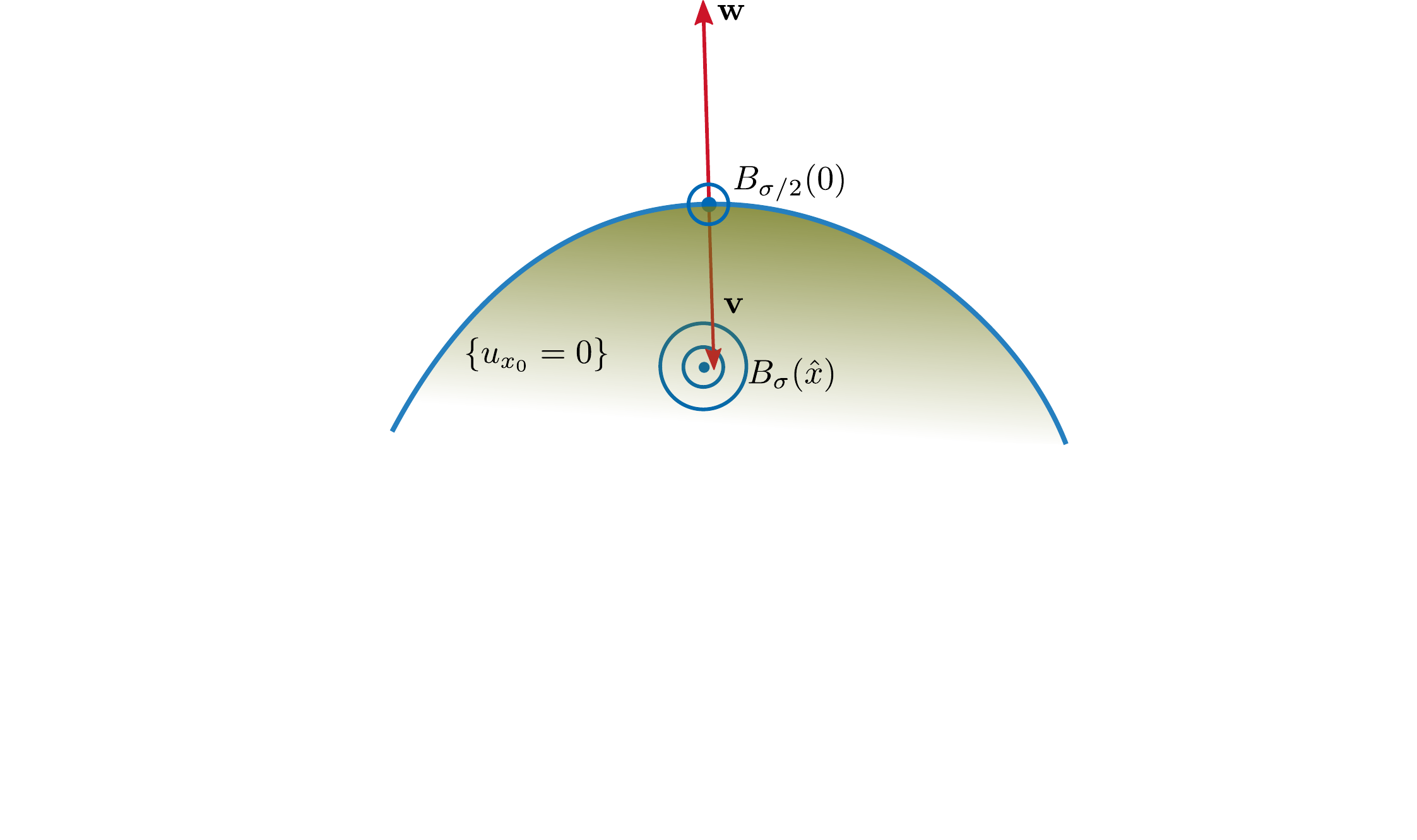}
\vspace{-4.3cm}
\caption{The contact set of $u_{x_0}$ when it has positive measure.}
\label{Fig-contact pos}
\end{figure}

We claim that
$$
\partial_{\mathbf w} u_{x_0}(x) \geq 0 \qquad \forall\,x \in B_{\sigma/2}(0).
$$
Indeed, given any point $x \in B_{\sigma/2}(0)$,
consider $\mathbf{v} \in B_{\sigma/2}(\hat x)$ and define the point
$$
y_x:=x+\mathbf{v} \in B_{\sigma}(\hat x).
$$
Note that, by the convexity
of $u_{x_0}$, it follows that
$$
\frac{d}{dt}\partial_{\mathbf w}u_{x_0}(y_x+t\mathbf w)=\frac{d^2}{dt^2}u_{x_0}(y_x+t\mathbf w)\geq 0\qquad \forall\, t \geq 0.
$$
Hence, since $\partial_{\mathbf w}u_{x_0}(y_x)=0$ (because $y_x \in B_\sigma(\hat x)\subset  \{u_{x_0}=0\}$) we get
$\partial_{\mathbf w} u_{x_0}(x)=\partial_{\mathbf w} u_{x_0}(y_x+|\mathbf v|\mathbf w)\geq 0$, as desired.

Observe now that the function $\partial_{\mathbf w}u_{x_0}$ is harmonic inside the set $\{u_{x_0}>0\}$.
In particular, unless it is identically zero
it must be strictly positive there, by the strong maximum principle. 
Thanks to this consideration and to the fact that
$\Delta u_{x_0}=1$ inside $\{u_{x_0}>0\}$,
one easily concludes that
\begin{equation}
\label{eq:monotone}
\partial_{\mathbf w} u_{x_0}(x)> 0 \qquad \forall\,x \in B_{\sigma/2}(0)\cap \{u_{x_0}>0\}.
\end{equation}
Observe now that
$$
0 \leq \frac{\partial_{\mathbf w} u_{x_0}(x)}{|\nabla u_{x_0}(x)|}=\frac{\nabla u_{x_0}(x)}{|\nabla u_{x_0}(x)|}\cdot \mathbf w,
$$
and that $\frac{\nabla u_{x_0}(x)}{|\nabla u_{x_0}(x)|}$ coincides with the normal to the level sets of $u_{x_0}$. Hence, recalling the definition of $\mathbf w$,
\eqref{eq:monotone}
implies that the normal to the level sets of $u_{x_0}$ inside $B_{\sigma/2}(0)$ belongs to the cone 
$$
\mathcal C:=\{\nu \in \mathbb S^{n-1}\,:\, \nu\cdot\mathbf v \leq  0  
\quad \forall\,\mathbf v\in B_{\sigma/2}(\hat x)\},
$$
see Figure \ref{Fig-monotone}.
\begin{figure}[ht]
\hspace*{-2cm}
\includegraphics[width=0.95\textwidth]{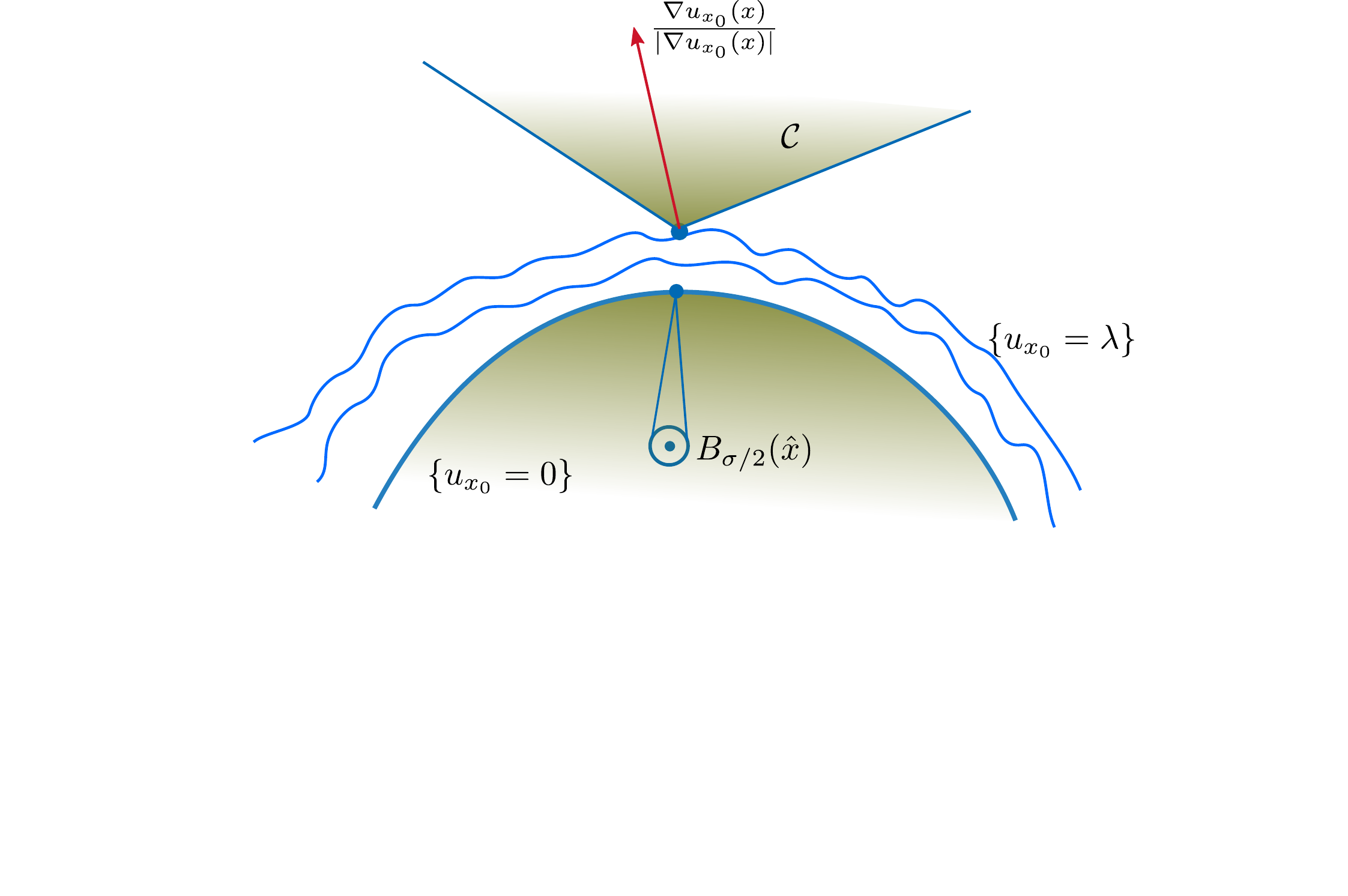}\vspace{-4cm}
\caption{Monotonicity of $u_{x_0}$ and Lipschitz regularity of its level sets.}
\label{Fig-monotone}
\end{figure}

Thus, if one looks at the level sets $\{u_{x_0}=\lambda\}_{\lambda>0}$ as graphs with respect to the hyperplane orthogonal to $\hat x$, it follows that these graphs have a bounded slope, or equivalently they are uniformly Lipschitz.
In particular, since $\partial\{u_{x_0}>0\}=\lim_{\lambda\to 0^+}\{u_{x_0}=\lambda\}$, 
the free boundary is Lipschitz continuous inside $B_{\sigma/2}(0)$.\footnote{One may note that the Lipschitz continuity of $\partial\{u_{x_0}>0\}$ is also an immediate consequence of the convexity of the contact set. However, as the reader will see in the sequel of Case 1, this argument can be generalized to prove the Lipschitz regularity of $\partial\{u_{x_0,r}>0\}$ for $r$ small enough.}

Thanks to this fact, we see that the functions $\partial_{\mathbf w} u_{x_0}$ are harmonic, strictly positive inside $B_{\sigma/2}(0)\cap \{u_{x_0}>0\}$, and vanish on the Lipschitz hypersurface $\partial\{u_{x_0}>0\}$.
Hence, by the boundary Harnack inequality for harmonic functions in Lipschitz domains (see for instance \cite[Theorem 11.6]{CS05}) it follows that 
$$
\frac{\partial_{\mathbf w_1} u_{x_0}}{\partial_{\mathbf w_2} u_{x_0}} \in C^{0,\alpha}(B_{\sigma/4}(0)\cap \{u_{x_0}>0\}) \qquad \forall\,\mathbf w_i=-\frac{\mathbf v_i}{|\mathbf v_i|},\,\mathbf v_i \in B_{\sigma/2}(\hat x).
$$
In particular, choosing 
$$
\mathbf w_2=-\frac{\hat x}{|\hat x|},\qquad
\mathbf w_1=-\frac{\hat x+\frac{\sigma}{2} \mathbf e}{|\hat x+\frac{\sigma}{2} \mathbf e|}\quad \text{with }\mathbf e \in \mathbb S^{n-1}\text{ arbitrary},\
$$
we get
 $$\frac{\nabla u_{x_0}(x)}{|\nabla u_{x_0}(x)|}\in C^{0,\alpha}(B_{\sigma/4}(0)\cap \{u_{x_0}>0\}).
 $$
Since $\frac{\nabla u_{x_0}(x)}{|\nabla u_{x_0}(x)|}$ coincides with the normal to the level sets of $u_{x_0}$, we deduce that the level sets of $u_{x_0}$ are uniformly $C^{1,\alpha}$, thus $\partial\{u_{x_0}>0\}$
is of class $C^{1,\alpha}$ inside $B_{\sigma/4}(0)$. In particular, up to reducing the size of $\sigma$, 
$\{u_{x_0}>0\}$ is arbitrarily close to a half-space $H$ inside $B_{\sigma/4}(0)$.

To transfer the informations back to our solution $u$, one argues as follows: given $\epsilon>0$ small, we can choose a radius $r=r_\epsilon>0$ small enough so that
$\|u_{x_0,r}-u_{x_0}\|_{C^1(B_1)}\leq \epsilon$. This implies that 
\eqref{eq:monotone} almost holds for $u_{x_0,r}$: more precisely, since \eqref{eq:monotone} holds uniformly away from the free boundary, we deduce that
$$
\partial_{\mathbf w} u_{r,x_0}(x)> 0 \qquad \forall\,x \in B_{\sigma/2}(0)\cap H\quad \text{s.t. } {\rm dist}(x,\partial H)\geq  \delta_\epsilon \sigma,
$$
where $\delta_\epsilon\to 0$ as $\epsilon\to 0.$

Exploiting the fact that $\partial_{\mathbf w} u_{r,x_0}$ is harmonic inside its positivity set, and that the set $\{u_{r,x_0}>0\}\cap B_{\sigma/2}$ is very close to $B_{\sigma/2}(0)\cap H$, a maximum principle argument (see \cite[Lemma 11]{C98}) shows that,
if $\epsilon$ is sufficiently small,
$$
\partial_{\mathbf w} u_{r,x_0}(x)> 0 \qquad \forall\,x \in B_{\sigma/4}(0)\cap \{u_{r,x_0}>0\}.
$$
Exactly as before, this implies first that the level sets of $u_{r,x_0}$ are Lipschitz, and then they are $C^{1,\alpha}$ inside $B_{\sigma/6}(0)$ by the boundary Harnack inequality.
Finally, elliptic PDEs techniques yield higher order regularity and analyticity of $\partial \{u_{r,x_0}>0\}\cap B_{\sigma/8}(0)$ \cite{KN77}.
Since $\partial \{u_{r,x_0}>0\}\cap B_{\sigma/8}(0)$ is a dilate and translate of $\partial\{u>0\}\cap B_{r\sigma/8}(x_0))$, this proves the analiticity of the free boundary of $u$ in a neighborhood of $x_0$.

\smallskip

\noindent
$\bullet$ {\it Case 2: The contact set $\{u_{x_0}=0\}$ has measure 0.} In this case, since $\Delta u_{x_0}=1$ outside the contact set we deduce that $\Delta u_{x_0}=1$ a.e. in $\R^n$, and by elliptic regularity $\Delta u_{x_0}\equiv 1$.
Thus
$$
0=\partial_{\mathbf e\mathbf e}1=\partial_{\mathbf e\mathbf e}(\Delta u_{x_0})
=\Delta (\partial_{\mathbf e\mathbf e}u_{x_0}),
$$
which implies that $\partial_{\mathbf e\mathbf e}u_{x_0} \geq 0$ is a nonnegative harmonic function in the whole $\R^n$. Recall now the classical Harnack inequality for nonnegative harmonic functions:
\begin{equation}
\label{eq:harnack}
\Delta w=0 \,\,\text{ and }\,\,w \geq 0 \quad \text{ in }B_R\qquad \Rightarrow\qquad \sup_{B_{R/2}}w \leq C_n\inf_{B_{R/2}}w,
\end{equation} for some dimensional constant $C_n>0.$

Defining $m_{\mathbf e}:=\inf_{\R^n}\partial_{\mathbf e\mathbf e}u_{x_0}$, it follows by \eqref{eq:harnack} applied to $w=\partial_{\mathbf e\mathbf e}u_{x_0} -m_{\mathbf e}$ that
$$
\sup_{B_{R/2}}(\partial_{\mathbf e\mathbf e}u_{x_0} -m_{\mathbf e}) \leq C_n\inf_{B_{R/2}}(\partial_{\mathbf e\mathbf e}u_{x_0} -m_{\mathbf e})\qquad\forall\,R>0.
$$
Letting $R\to \infty$ the right hand side tends to 0, therefore
$$
\sup_{\R^n}(\partial_{\mathbf e\mathbf e}u_{x_0} -m_{\mathbf e})=0,
$$
which proves that $\partial_{\mathbf e\mathbf e}u_{x_0}$ is constant.

Since $\mathbf e\in \mathbb S^{n-1}$ was arbitrary, this proves that all the pure second derivatives of $u_{x_0}$ are constant, hence $u_{x_0}$ is a homogeneous quadratic polynomial (recall that $u_{x_0}(0)=0$ and $u_{x_0}\geq 0$). Thus, $x_0$ is a singular point.

We now observe that $u_{x_0}=\lim_{k\to \infty}u_{x_0,r_k}$ for some sequence of radii $r_k$ converging to 0. Now, if we consider a different sequence of radii $r_{k'}$ converging to 0 and $u_{x_0}'$ is a  limit point, then again the contact set $\{u_{x_0}'=0\}$ must have measure zero. Indeed, if not, then by Case 1 the point $x_0$ would be regular. Thus, the set $\{u_{x_0,r}=0\}$ would be close to a half-ball for all $r>0$ small, in contradiction with the fact that the set  $\{u_{x_0,r_k}=0\}$ should converge to $\{u_{x_0}=0\}$, which is of measure zero.

Finally, the inclusion  $\partial\{u>0\}\cap B_{r}(x_0)\subset \{x\,:\,|\mathbf{e}_r\cdot (x-x_0)|\leq o(r)\}$ follows from the fact that
$$
\frac1r\Bigl(\bigl(\partial\{u>0\}\cap B_{r}(x_0)\bigr) -x_0\Bigr)=\partial\{u_{x_0,r}>0\}\cap B_{1},
$$
and that any limit point of $u_{x_0,r}$ has a contact set which is a convex set of measure zero, hence contained in a hyperplane.

\smallskip

We refer the interested reader to \cite{C98} and to the book \cite[Chapters 4.1, 6.1, 6.2, and 6.4]{PSU12} for more details and specific references.
\end{proof}

Observe that, as a consequence of Theorem \ref{thm:dico}, a free boundary point can only be either regular or singular. Also, if it is regular then the free boundary is smooth in a neighborhood and all points nearby are regular as well.
From this we deduce that the convergence in Definition \ref{def:reg pt} holds without the need of taking a subsequence of radii.

While Theorem \ref{thm:dico}(i) gives a  complete answer to the structure of regular points, Theorem \ref{thm:dico}(ii) is still not conclusive. Indeed, first of all the vector $\mathbf{e}_r$ may depend on $r$. Also, the $o(r)$ appearing in the statement comes from a compactness argument, so it is not quantified.

Hence, from now on we shall only focus on the study of singular points.
To simplify the notation we  denote
$$\Sigma:=\{\text{singular points}\}\subset \partial\{u>0\}.$$
Since the set of regular points is relatively open inside the free boundary (as a consequence of Theorem \ref{thm:dico}(i)), we deduce that $\Sigma$ is a closed set.

\section{Uniqueness of blow-up at singular points}

As observed in the previous section, a priori the vector $\mathbf{e}_r$ appearing in the statement of
 Theorem \ref{thm:dico}(ii)
may depend on $r$. This fact is essentially related to whether the convergence in Definition \ref{def:sing pt} holds only up to subsequence or not: indeed, if one could prove that the convergence to a polynomial $p$ holds without passing to a subsequence, then one could easily deduce that 
$\partial\{u>0\}\cap B_{r}(x_0)\subset \{x\,:\, {\rm dist}(x-x_0,\{p=0\})\leq o(r)\}$.

The complete answer to this questions has again been given by Caffarelli in \cite{C98}, after some previous results in the two dimensional case \cite{CR77,Sak91,Sak93}.

From now on we use the notation 
\[\mathcal P := \bigg\{ p(x) = { \frac 1 2} \langle  Ax,x\rangle\  :\   A\in \R^{n\times n}\mbox{ symmetric nonnegative definite, with ${\rm tr}(A)=1$}
\bigg\}
.\]

\begin{theorem}
\label{thm:uniq blow 2}
Let $u$ solve \eqref{eq:obst},
and let $x_0\in \Sigma$. Then
there exists $p_{*,x_0} \in \mathcal P$ such that
$$
\lim_{r\to 0}\frac{u(x_0+rx)}{r^2}= p_{*,x_0}(x).
$$
In addition, the map
$$
\Sigma\ni x_0\mapsto p_{*,x_0}
$$
is locally uniformly continuous.
\end{theorem}

For convenience of the reader we present here the proof of this result given few years later by Monneau \cite{M03}.
To this aim, we shall first need to prove the following monotonicity formula due to Weiss \cite{W99}.

\begin{proposition}
\label{prop:Weiss}
Let $0 \in \partial\{u>0\}$, assume that $B_\rho(0)\subset \Omega$,
and for $r\in(0,\rho)$ define
the function
\[
W(r,u) := \frac{1}{r^{n+2}} \int_{B_r} \Bigl(|\nabla u|^2 +2u\Bigr)  -\frac{2}{r^{n+3}}\int_{\partial B_r} u^2. 
\]
Then
\[
\frac{d}{dr} W(r,u) \ge 0\qquad \forall\,r\in (0,\rho).
\]
In addition, if $0 \in \Sigma$ then
\[
W(0^+, u) =W(r,p)\qquad \forall\,p \in \mathcal P,\,\forall \,r>0.
\]
\end{proposition}
\begin{proof}
Set $u_r(x):=r^{-2}u(rx)$,
so that
$W(r,u)=W(1,u_r)$.
Then, integrating by parts,
\begin{align*}
\frac{d}{dr} W(1,u_r)&=2\int_{B_1}\Bigl(\nabla u_r\cdot \nabla (\partial_ru_r)+\partial_ru_r\Bigr)-
4\int_{\partial B_1}u_r\,\partial_ru_r\\
&=2\int_{B_1}\bigl(-\Delta u_r+1\bigr)\,\partial_ru_r+
2\int_{\partial B_1}(\partial_\nu u_r-2u_r)\,\partial_ru_r.
\end{align*}
 Noticing that
$\Delta u_r=1$ in the region where $\{u_r>0\}$, and that
$$
\partial_ru_r=r^{-1}\bigl(x\cdot \nabla u_r-2u_r\bigr),
$$
it follows that
either $-\Delta u_r+1$ or $\partial_ru_r$ vanishes, and that 
 $\partial_ru_r=r^{-1}(\partial_\nu u_r-2u_r)$ on $\partial B_1$. Thus $(-\Delta u_r+1)\partial_ru_r\equiv 0$, and 
$$
\frac{d}{dr} W(1,u_r)
=\frac{d}{dr} W(1,u_r)
=\frac2{r}\int_{\partial B_1}(\partial_\nu u_r-2u_r)^2 \geq 0.
$$
This proves the monotonicity of $W$.

Now, if $0$ is a singular point and $\bar p \in \mathcal P$ is the limit of $u_{r_k}$ along some sequence $r_k\to 0$, 
then 
$$
\lim_{r\to 0}W(r,u)=\lim_{k\to \infty}
W(r_k,u)=\lim_{k\to \infty}
W(1,u_{r_k})=
W(1,\bar p),
$$
where the first equality follows from the fact that $W(r,u)$ has a limit as $r\to 0$ because of monotonicity.
Finally, a direct computation shows that there exists a dimensional constant $c_n>0$ such that 
$W(r,p)=W(1,p)=c_n$ for all $p \in \mathcal P$ and $r>0$.
\end{proof}

We shall also need the following observation:
\begin{remark}\label{remsign} 
Let $p\in \mathcal P$. Since  $\Delta u = \Delta p =1$  in $\{u>0\}$, we have
\[
w\Delta w = 
\begin{cases} 
0  							&\mbox{in } \{u>0\}
\\
p\Delta p = p\ge 0 \quad 	&\mbox{in } \{u=0\}.
\end{cases}
\]
Equivalently,
\begin{equation}\label{wLapw}
w\Delta w = p\chi_{\{u=0\}} \ge 0\qquad \forall\, p \in \mathcal P.
\end{equation}
\end{remark}

We can now prove the so-called Monneau's monotonicity formula.
\begin{lemma}\label{lem:Mon}
Let $0 \in \Sigma$, $p \in \mathcal P$,
assume that $B_\rho(0)\subset \Omega$, and for $r \in (0,\rho)$ define
$$
M(r,u,p):=\frac{1}{r^{n+3}}
 \int_{\partial B_r} (u-p)^2.
$$
Then
$$
\frac{d}{dr}M(r,u,p)\geq 0\qquad \forall\,r \in (0,\rho).
$$
\end{lemma}
\begin{proof}
For simplicity of notation we set $w:=u-p$.

Since  $W(r,u)\geq W(0^+,u) =  W(r,p)$ for all $r\in (0,\rho)$ (see Proposition \ref{prop:Weiss}) and $\Delta p\equiv 1$, we have
\[
\begin{split}
0 &\le W(r,u) - W(r,p)
\\
&=\frac{1}{r^{n+2}}\int_{B_r} \Bigl(|\nabla u|^2 +2u -|\nabla p|^2-2p\Bigr) -\frac{2}{r^{n+3}}\int_{\partial B_r} \Bigl(u^2-p^2\Bigr)
\\
&=\frac{1}{r^{n+2}}\int_{B_r} \Bigl(|\nabla w|^2 + 2\nabla w\cdot \nabla p +2w\Bigr) -\frac{2}{r^{n+3}}\int_{\partial B_r} \Bigl(w^2 + 2w p\Bigr)
\\
&= \frac{1}{r^{n+2}}\int_{B_r} \Bigl(|\nabla w|^2 +2{\rm div}(w\nabla p)\Bigr) -\frac{2}{r^{n+3}}\int_{\partial B_r} \Bigl(w^2 + 2w p\Bigr)
\\
&= \frac{1}{r^{n+2}}\int_{B_r} |\nabla w|^2  -\frac{2}{r^{n+3}}\int_{\partial B_r} w^2 + \frac{2}{r^{n+3}}\int_{\partial B_r}   w(x\cdot\nabla p -2p)
\\
&= \frac{1}{r^{n+2}}\int_{B_r} |\nabla w|^2 -\frac{2}{r^{n+3}}\int_{\partial B_r} w^2 ,
\end{split}
\]
where we used that $x=r\nu$ on $\partial B_r$, and that $p$ is $2$-homogeneous (hence $x\cdot \nabla p=2p$).  This proves that
\begin{equation}
\label{eq:mon1}
\frac{1}{r^{n+2}}\int_{B_r} |\nabla w|^2\geq \frac{2}{r^{n+3}}\int_{\partial B_r} w^2\qquad \forall\,r\in (0,\rho).
\end{equation}
Now, since
$$
\frac{1}{r^{n+2}}\int_{B_r} |\nabla w|^2=
 \frac{1}{r^{n+2}}\int_{B_r} - w\Delta w + \frac{1}{r^{n+3}}\int_{\partial B_r} w\,x\cdot \nabla w,
$$
thanks to \eqref{eq:mon1}
and \eqref{wLapw} we obtain
$$
 \frac{1}{r^{n+3}} \int_{\partial B_r} w( x\cdot\nabla w-2w) \ge \frac{1}{r^{n+2}}\int_{B_r} w\Delta w\ge 0 
 \qquad \forall\,r\in (0,\rho).
$$
Hence, setting $w_r(x):=r^{-2}w(rx)$
and noticing that $\partial_rw_r(x)=r^{-1}(x\cdot \nabla w_r-2w_r)$, we obtain
\begin{multline*}
\frac{d}{dr}M(r,u,p)=\frac{d}{dr}\biggl(\frac{1}{r^{n+3}}
 \int_{\partial B_r} w^2\biggr)=\frac{d}{dr}\biggl(
 \int_{\partial B_1} w_r^2\biggr)\\
 =\frac{2}{r}
 \int_{\partial B_1}w_r( x\cdot\nabla w_r-2w_r) =
 \frac{2}{r^{n+4}} \int_{\partial B_r} w( x\cdot\nabla w-2w)\geq 0 ,
\end{multline*}
as desired.
\end{proof}

We can now prove the uniqueness and the continuity of blow-ups at singular points:
\begin{proof}[Proof of Theorem \ref{thm:uniq blow 2}]
We first prove the existence of the limit.

Assume with no loss of generality that $x_0=0$,
set $u_r(x):=r^{-2}u(rx)$, and let 
$p_1$ and $p_2$ be two different limits obtained along two sequences $r_{k,1}$ and $r_{k,2}$ both converging to zero. Up to taking a subsequence of $r_{k,2}$ and relabeling the indices, we can assume that $r_{k,2}\leq r_{k,1}$ for all $k$.
Thus, thanks to Lemma \ref{lem:Mon}, we have
$$\int_{B_1}(u_{r_{k,1}}-p_1)^2=M(r_{k,1},u,p_1)\geq M(r_{k,2},u,p_1)=\int_{B_1}(u_{r_{k,2}}-p_1)^2\qquad \forall\,k,
$$
and letting $k\to \infty$ we obtain
$$
0=\lim_{k\to \infty} \int_{B_1}(u_{r_{k,1}}-p_1)^2\geq \lim_{k\to \infty}\int_{B_1}(u_{r_{k,2}}-p_1)^2 =\int_{B_1}(p_2-p_1)^2.
$$
This proves that there is a unique possible limit for $u_r$ as $\to 0$, which implies that the limit exists.
From now on, given a singular point $x_0$, we shall denote this limit by $p_{*,x_0}$.

We now prove the continuity of the map $x_0 \mapsto p_{*,x_0}$ at $0 \in \Sigma$.
Fix $\ep>0$, and consider a sequence $x_k\in \Sigma$ with $x_k\to 0$.
Since $u_r \to p_{*,0}$, there exists a small radius $r_\ep>0$ such that
\begin{equation}
\label{eq:eps}
\int_{\partial B_1}\biggl|\frac{u(r_\ep x)}{r_\ep^2}-p_{*,0}(x)\biggr|^2 \leq \ep.
\end{equation}
Also, applying Lemma \ref{lem:Mon} at $x_k$ with $p=p_{*,0}$, we deduce that
\begin{align*}
\int_{\partial B_1}|p_{x_k,*} - p_{*,0}|^2 &=\lim_{r\to 0}
 \int_{\partial B_1}\biggl|\frac{u(x_k+r x)}{r^2} - p_{*,0}(x)\bigg|^2\\
 & \leq 
  \int_{\partial B_1}\biggl|\frac{u(x_k+r_\ep x)}{r_\ep^2} - p_{*,0}(x)\bigg|^2.
\end{align*}
Hence, letting $k \to \infty$ and recalling \eqref{eq:eps} we obtain
$$
\limsup_{k\to \infty}\int_{\partial B_1}
|p_{x_k,*} - p_{*,0}|^2\leq \lim_{k\to \infty} \int_{\partial B_1}\biggl|\frac{u(x_k+r_\ep x)}{r_\ep^2} - p_{*,0}(x)\bigg|^2 \le	 \ep.
$$
Since $\ep>0$ is arbitrary, this proves the continuity at $0$.

Because $\Sigma$ is locally compact (recall that $\Sigma$ is closed), this actually implies that the map
$$
\Sigma\ni x_0\mapsto p_{*,x_0}
$$
is locally uniformly continuous.
\end{proof}

\section{Stratification and $C^1$ regularity of the singular set}
With Theorem \ref{thm:uniq blow 2} at hand, we can now investigate the regularity of $\Sigma$.
Note that singular points may look very different depending on the dimension of the set $\{p_{*,x_0}=0\}$, see Figures 
\ref{Fig-blow2bis} and
\ref{Fig-blow2bis3}.
This suggests to stratify the set of singular points according to this dimension.

More precisely,
given
$x_0\in \Sigma$ we set
$$
k_{x_0}:={\rm dim}({\rm ker} \,D^2p_{*,x_0})
={\rm dim}(\{p_{*,x_0}=0\}).
$$
Then, given $m\in \{0,\ldots,n-1\}$ we define
$$
\Sigma_m:=\{x_0\in \Sigma\,:\,k_{x_0}=m\}.
$$
Note that, with this definition, the point in Figure \ref{Fig-blow2bis} belongs to $\Sigma_2$, while the point in Figure \ref{Fig-blow2bis3} belongs to $\Sigma_1$.

More in general, $\Sigma_0$ consists of isolated points, while the other strata $\Sigma_m$ should correspond to the $m$-dimensional part of $\Sigma$, see Figure \ref{Fig-contact}.

\begin{figure}[ht]
\includegraphics[width=0.6\textwidth]{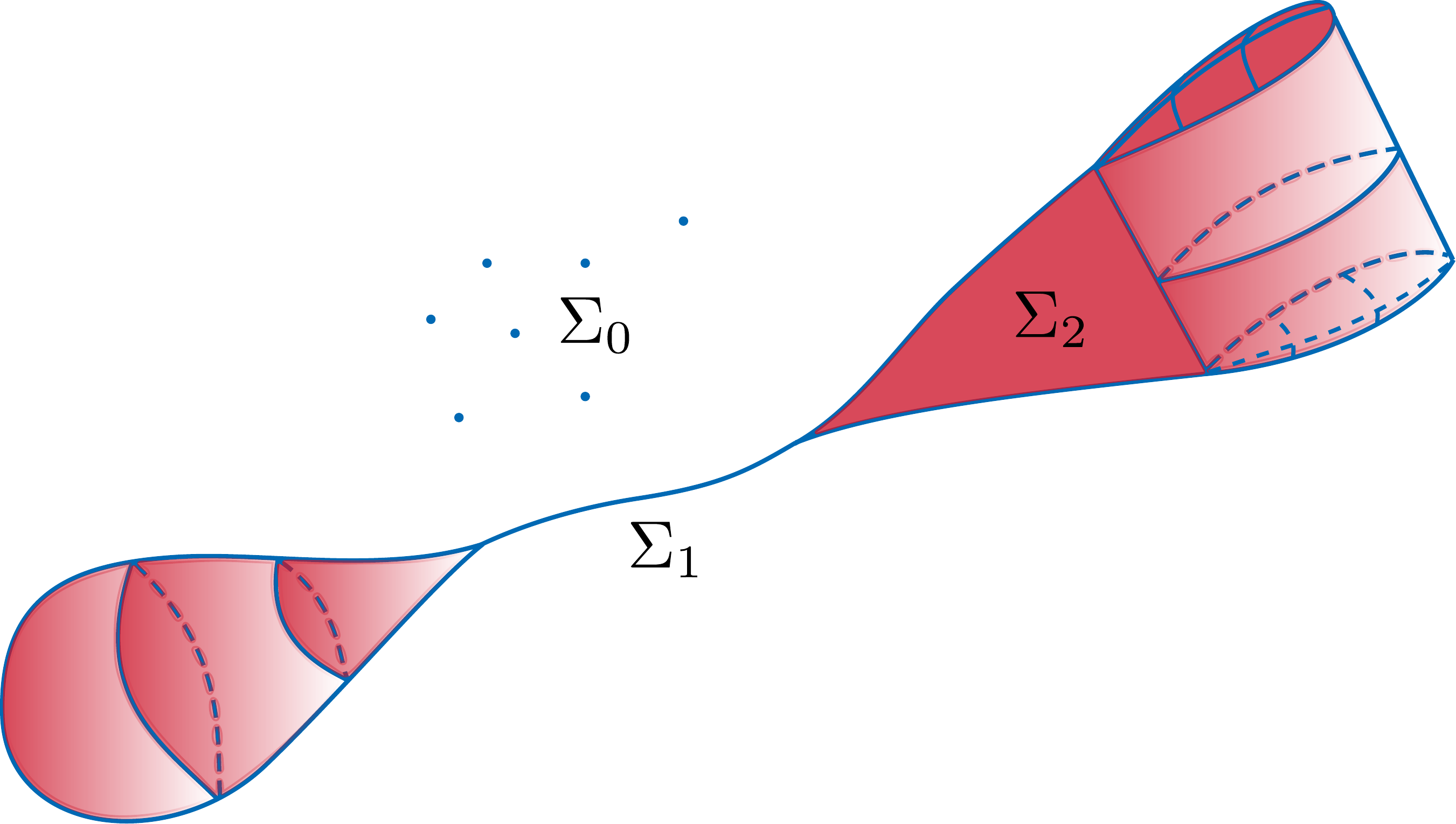}
\caption{A possible example of contact set in 3 dimensions.}
\label{Fig-contact}
\end{figure}

This intuition is confirmed by the following result due to Caffarelli \cite{C98}:
\begin{theorem}
For any  $m\in \{0,\ldots,n-1\}$,
$\Sigma_m$ is locally contained in 
a $m$-dimensional manifold of class $C^1$.
\end{theorem}

\begin{proof}[Idea of the proof]
Recalling that $u=\nabla u\equiv 0$ on the contact set $\{u=0\}$ (see Figure \ref{Pic-ell-ob}), 
we have in particular that $u|_{\Sigma_m}=\nabla u|_{\Sigma_m}\equiv 0$.
Also, by Theorem \ref{thm:uniq blow 2},
$$
u(x_0+y)=p_{*,x_0}(y)+o(|y|^2).
$$
Hence, at least formally, this means that $p_{*,x_0}$ corresponds to the second order term in the Taylor expansion of $u$,
namely ``$p_{*,x_0}(y)=\frac12 D^2u(x_0)[y,y]$''.

Since the map
$\Sigma\ni x_0\mapsto p_{*,x_0}(y)$ is continuous,
%
one can apply Whitney's extension theorem to find a map $F:\R^n\to \R^n$ of class $C^1$ such that
$$
F(x_0)=\nabla u(x_0)=0 
\qquad \text{and}\qquad \nabla F(x_0)=D^2p_{*,x_0}\qquad \forall\,x_0\in \Sigma_m.
$$
Noticing that ${\rm dim}({\rm ker }\,\nabla F(x_0))={\rm dim}({\rm ker }\,D^2p_{*,x_0})=m$ on $\Sigma_m$, it follows by the Implicit Function Theorem that
$$
{\Sigma_m}=\{F=0\}\cap {\Sigma_m}
$$
is locally contained in 
a $C^1$ $m$-dimensional manifold,
 as desired.
\end{proof}

\begin{remark}
The proof above shows that the estimate
\begin{equation}
\label{eq:C2 u}
\|u(x_0+\cdot)-p_{*,x_0}\|_{L^\infty(B_r)}= o(r^2),
\end{equation}
with a bound $o(r^2)$ independent of $x_0$, implies that 
$\Sigma_m$ is locally contained in a $C^1$ $m$-dimensional manifold. 

More in general, if one could prove that
\begin{equation}
\label{eq:C2a u}
\|u(x_0+\cdot)-p_{*,x_0}\|_{L^\infty(B_r)}\leq C\,r^{2+\alpha}
\end{equation}
for some constant $C$ independent of $x_0$, then by applying Whitney's extension theorem in H\"older spaces one would conclude that $\Sigma_m$ is contained in a $m$-dimensional manifold of class $C^{1,\alpha}$.
\end{remark}

\begin{remark}
The fact that $\Sigma_m$ is only contained in a manifold (and does not necessarily coincide with it) is optimal: one can build examples in $n=2$ where $\Sigma_1$ coincides with a Cantor set contained in a line \cite{Sch76}.
\end{remark}

\section{Recent developments}
\label{sect:recent}

In 1999, using Proposition \ref{prop:Weiss} combined with 
a ``epiperimetric  approach'', Weiss proved the following result \cite{W99}:
\begin{theorem}
\label{thm:Weiss}
Let $n=2$ and $x_0\in \Sigma$. Then there exists a constant $\alpha>0$ such that
$$
\|u(x_0+\cdot)-p_{*,x_0}\|_{L^\infty(B_r)}
\leq
C \,
r^{2+\alpha}
$$
where $C>0$ is locally independent of $x_0$.
In particular $
\Sigma_1\subset C^{1,\alpha}$ curve.
\end{theorem}

Weiss' proof was restricted to two dimensions because of some delicate technical arguments in some steps of the proof. Still, one could have hoped to push his argument to higher dimensions. This was achieved last year by Colombo, Spolaor, and Velichkov \cite{CSV17}, where the authors introduced a quantitative argument to avoid a compactness step in Weiss' proof. However, the price to pay for working in higher dimensions is that they can only get a logarithmic improvement in the convergence of $u$ to $p_{*,x_0}$:
\begin{theorem}
\label{thm:CSV}
Let $n\geq 3$ and $x_0\in \Sigma$. Then
exists a dimensional constant $\epsilon>0$ such that
$$
\|u(x_0+\cdot)-p_{*,x_0}\|_{L^\infty(B_r)}
\leq
C \,r^2|\log(r)|^{-\epsilon}
$$
where $C>0$ is locally independent of $x_0$.
In particular $
\Sigma_m\subset C^{1,\log^\epsilon}$ $m$-dim manifold.
\end{theorem}
In other words, in dimension $n\geq 3$ one can improve the $C^1$ regularity of Caffarelli to a quantitative one, with a logarithmic modulus of continuity. This result raises the question of whether one may hope to improve such an estimate, or if this logarithmic bound is optimal.

In a recent paper with Serra \cite{FSell} we showed that, up to the presence of some ``anomalous'' points of higher codimension where \eqref{eq:C2a u} is false for any $\alpha>0$, one can actually prove that \eqref{eq:C2a u} holds with $\alpha=1$ at most points (these points will be called ``generic'').
In particular, up to a small set, singular points can be covered by $C^{1,1}$ (and in some cases $C^2$) manifolds. As we shall discuss in Remark \ref{rmk:optimal} below, this result provides the optimal decay estimate for the contact set. In addition, we can prove that anomalous points may exist and our bound on their Hausdorff dimension is optimal.
In particular, the existence of anomalous points implies that also Theorem \ref{thm:CSV} is optimal.

Before stating our result we note that, as a consequence of Theorem \ref{thm:uniq blow 2}, points in $\Sigma_0$ are isolated and $u$ is strictly positive in a neighborhood of them. In particular $u$ solves the equation $\Delta u=1$ in a neighborhood of $\Sigma_0$, hence it is analytic there. Thus, it is enough to understand the structure of $\Sigma_m$ for $m=1,\ldots,n-1$.

Here and in the sequel, ${\rm dim}_{\mathcal H}(E)$ denotes the Hausdorff dimension of a set $E$.
The main result in \cite{FSell} the following:
\begin{theorem}
\label{thm:main}
Let $\Sigma:=\cup_{m=0}^{n-1}\Sigma_m$ denote the set of singular points. Then:
\begin{enumerate}
\item[($n=2$)] $\Sigma_1$ is locally  contained in a $C^{2}$ curve.
\item[($n\ge 3$)] \begin{enumerate}
\item
The higher dimensional stratum $\Sigma_{n-1}$ 
can be written as the disjoint union of ``generic points'' $\Sigma_{n-1}^g$ and ``anomalous points'' $\Sigma_{n-1}^a$, where:\\
- $\Sigma^g_{n-1}$ is locally contained in a $C^{1,1}$ $(n-1)$-dimensional manifold;\\
- $\Sigma^a_{n-1}$ is a relatively  open subset of $\Sigma_{n-1}$ satisfying  $${\rm dim}_{\mathcal H}(\Sigma^a_{n-1})\leq n-3$$  (actually, $\Sigma^a_{n-1}$ is discrete when $n=3$).

Furthermore, the whole stratum $\Sigma_{n-1}$ can be locally covered by a $C^{1,\alpha_0}$  $(n-1)$-dimensional manifold, for some dimensional exponent $\alpha_0>0$. 
\item For all $m=1,\ldots,n-2$ we can write $\Sigma_{m}=\Sigma_m^g\cup\Sigma_g^a$,
where:\\
- $\Sigma_m^g$ can be locally covered by a $C^{1,1}$
$m$-dimensional manifold;\\
- $\Sigma^a_m$ is a relatively open subset of $\Sigma_{m}$ satisfying $${\rm dim}_{\mathcal H}(\Sigma^a_m)\leq m-1$$ (actually, $\Sigma^a_m$ is discrete when $m=1$).

In addition, the whole stratum $\Sigma_{m}$ can be locally covered by a $C^{1,\log^{\epsilon_0}}$ $m$-dimensional manifold, for some dimensional exponent $\epsilon_0>0$. 
\end{enumerate}
\end{enumerate}
\end{theorem}

This result needs several comments.

\begin{remark}\label{rmk:optimal}
We first discuss the optimality of the theorem above.
\begin{enumerate}
\item
Our $C^{1,1}$ regularity provides the optimal control on the contact set in terms of the density decay. Indeed our result implies that, at all singular points up to a $(n-3)$-dimensional set (in particular at all singular points when $n=2$, and at all singular points up to a discrete set when $n=3$), the following bound holds:
$$
\frac{|\{u=0\}\cap B_r(x_0)|}{|B_r(x_0)|}\leq Cr\qquad \forall\,r>0.
$$
In view of the two dimensional Example 1 in \cite[Section 1]{Sch76}, this estimate is optimal.

\item The possible presence of anomalous points comes from different reasons depending on the dimension of the stratum.
\begin{enumerate}
\item The possible presence of points in $\Sigma_{n-1}^a$ comes from the potential existence, in dimension $n\ge 3,$ of $\lambda$-homogeneous solutions to the so-called Signorini problem with $\lambda\in(2,3)$. More precisely, it follows by our proof that the following result holds:

{\em Let $q\in W^{1,2}_{\rm loc}(\R^k)$ satisfy:\\
- $q$ is $\lambda_*$-homogeneous, namely $q(\varrho x)=\varrho^{\lambda_*}q(x)$ for all $\varrho>0$;\\ 
- $\Delta q\leq 0$, $q\Delta q\equiv 0$, $q|_{\{x_k=0\}}\geq 0$;\\ 
- $\Delta q=0$ inside $\R^k\setminus\{x_k=0\}$.\\
Let $k$ be the smallest dimension for which there exists a nontrivial solution to the problem above with $\lambda_* \in (2,3)$.
Then ${\rm dim}_{\mathcal H}(\Sigma^a_{n-1})\leq n-k$.}

Because  $k\geq 3$ is the best lower bound currently known on $k$ (see for instance \cite{FS17}), we get ${\rm dim}_{\mathcal H}(\Sigma^a_{n-1})\leq n-3$.

\smallskip

 \item For $m\leq n-2$, the anomalous points in the strata $\Sigma^a_m$ come from the possibility that, around a singular point $x_0$, the function $(u-p_{*,x_0})|_{B_r(x_0)}$ behaves as $\varepsilon_r q$, where:\\
 - $\varepsilon_r\in \R^+$ is infinitesimal as $r\to 0^+$, but $\varepsilon_r\gg r^\alpha$ for any $\alpha>0$;\\
- $q$ is a nontrivial second order harmonic polynomial.\\
 Although this behavior may look strange (we are saying that, after one removes  from $u$ its second order Taylor expansion, one still sees a second order polynomial), this can  actually happen and our estimate on the size of $\Sigma_m^a$ is optimal.
 Indeed, we can construct examples of solutions for which ${\rm dim}(\Sigma_m^a)=m-1$. 
\end{enumerate} 
\end{enumerate}
\end{remark}

We now make some general observations on Theorem \ref{thm:main}.
\begin{remark}
\label{rmk:main thm}
\begin{enumerate}
\item 
Our result extends Theorem \ref{thm:Weiss} to the highest dimensional stratum in every dimension, and improves it when $n = 2$.
\item
The last part of the statement in the case ($n\geq 3$)-(b) corresponds to Theorem \ref{thm:CSV}. In \cite{FSell} we obtain this result as a simple consequence of our analysis.
\item
The set of generic points is not open inside the singular set. In particular,  anomalous points can accumulate at generic points
even in dimension 3 (where anomalous points are discrete).
\item In \cite{Sak91,Sak93}, Sakai proved very strong structural results for the free boundary in dimension $n=2$. However, his results are very specific to the two dimensional case with analytic right hand side, as they rely on complex analysis techniques. On the other hand, all the results mentioned before \cite{C77,C98,W99,CSV17} are very robust and apply to more general right hand sides. Analogously, also our techniques are robust and can be extended to general right hand sides.
 \end{enumerate}
 \end{remark}

\noindent
{\bf Strategy of the  proof of Theorem \ref{thm:main}.}
The idea of the proof is the following:
as mentioned before, to obtain $C^{1,1}$ regularity of the singular set we would like to show that \eqref{eq:C2a u} holds with $\alpha\geq 1$.

So, let $0$ be a singular free boundary point. 
Using Weiss' and Monneau's motononicity formulae, we are able to prove that the so-called Almgren frequency function is monotone on $w:=u-p_*$. More precisely, if we set
$$
\tilde w_r(x):=\frac{w(rx)}{\|w(r\,\cdot\,)\|_{L^2(\partial B_1)}},
$$
then we can show
\begin{equation}
\label{eq:Almgren}
\frac{d}{dr}\|\nabla \tilde w_r\|_{L^2(B_1)}^2\geq \frac2r \biggl(\int_{B_1}\tilde w_r\Delta\tilde w_r\biggr)^2 \geq 0,
\end{equation}
see \cite[Proposition 2.4 and Equation (2.20)]{FSell}.

Note that, as a consequence of \eqref{eq:Almgren}, it follows that $\|\nabla \tilde w_r\|_{L^2(B_1)} \leq \|\nabla \tilde w_1\|_{L^2(B_1)}$ for all $r\leq 1$,
which allows us to perform blow-ups around $0$ by considering weak $W^{1,2}$ limits of $\tilde w_r$ as $r\to 0.$

Set $\lambda_*:=\lim_{r\to 0}\|\nabla \tilde w_r\|_{L^2(B_1)}^2$.
Using \eqref{eq:Almgren} again we can prove that, up to a subsequence,
$w_r {\rightharpoonup}q$ in $W^{1,2}(B_1)$,
where:\\ 
- $q$ is $\lambda_*$-homogeneous and $q \Delta q\equiv 0$;\footnote{Indeed, one can prove that   $q$ satisfies
$$
\frac{d}{dr}\|\nabla \tilde q_r\|_{L^2(B_1)}^2\equiv 0\qquad \text{where}\quad\tilde q_r(x):=\frac{q(rx)}{\|q(r\,\cdot\,)\|_{L^2(\partial B_1)}}
$$
(see the proof of \cite[Proposition 2.10]{FSell}).
Then, it follows from  \eqref{eq:Almgren} and its proof  that $q$ is $\lambda_*$-homogeneous with $\lambda_*=\|\nabla q\|_{L^2(B_1)}^2=\lim_{r\to 0}\|\nabla \tilde w_r\|_{L^2(B_1)}^2$, and that $q\Delta q \equiv 0$.
}\\ 
- $\Delta q\leq 0$, and $\Delta q$ is supported on $L=\{p_{*,0}=0\}$;\footnote{The fact that $\Delta q\leq 0$ follows by noticing that $\Delta w=\chi_{\{u>0\}}-1\leq 0$, thus $\Delta \tilde w_r \leq 0$ for all $r>0$. 

For the second part, note that $\Delta w=0$ inside $\{u>0\}$.
Therefore
$\Delta \tilde w_r=0$ inside the set $\frac1r\bigl(\{u>0\}\cap B_r)$, which converges to $B_1\setminus \{p_{*,0}=0\}$ as $r\to 0$.}

In addition, by a variant of the argument in the proof of Lemma \ref{lem:Mon} we can prove that
$$
\frac{d}{dr}\biggl(\frac{1}{r^{n-1+2\lambda_*}}
 \int_{\partial B_r} w^2\biggr)\geq 0.
$$
This implies that 
$$
\biggl(\frac{1}{r^{n-1}}\int_{\partial B_r}w^2\biggr)^{1/2} \leq \biggl(\int_{\partial B_1}w^2\biggr)^{1/2}r^{\lambda_*}\leq C\,r^{\lambda_*},
$$
from which we are able to deduce that
\eqref{eq:C2a u} holds at $x_0=0$ with
$
2+\alpha=\lambda_*.
$
Hence our problem is reduced to understanding the possible values of $\lambda_*$, and more precisely in proving (if possible) that $\lambda_*\geq 3$.

Although it is not difficult to prove that $\lambda_*\geq 2$ (this follows from the fact that $\|\nabla \tilde w_r\|_{L^2(B_1)}^2 \geq 2$, see \eqref{eq:mon1}), it is actually unclear how to exclude that $\lambda_*=2$. Note that, in the latter case, we would get no new information with respect to what was already known!

Recalling that $m={\rm dim}(L)$, we need to distinguish between the two cases $m=n-1$ and $m\leq n-2$. We begin with the latter.

\begin{itemize}
\item[$\bullet$] {\it The case $m\leq n-2$.}
Because $q \in W^{1,2}(B_1)$ and its Laplacian is concentrated on $L$ which has dimension  at most $n-2$, it follows by a classical capacity argument that $\Delta q$ must be identically zero, so 
 $q$ is harmonic. In particular $$\lambda_* \in \{2,3,4,\ldots\}.$$
 
Hence, we only need to exclude that $\lambda_*=2$ (this would correspond to the point $0$ being ``anomalous'').
Unfortunately, as already mentioned before, anomalous points may exists in dimension $n \geq 3$.
To circumvent this difficulty, a key ingredient in our analysis comes from the following fundamental relation, that we prove as a consequence of Lemma \ref{lem:Mon}:
\begin{equation}\label{howisD2q}
\int_{\partial B_1} q(p_*-p) \ge  0 \qquad \mbox{for all } p \in \mathcal P.
\end{equation}
Thanks to this inequality we can show that, whenever
$\lambda_*=2$, some very strong relation between $p_*$ and $q$ holds (see \cite[Proposition 2.10]{FSell}). Thus, our goal becomes proving that \eqref{howisD2q} cannot be true at ``too many'' singular points.

One of our key results shows that, when $n=3$ and $m=1$, then   \eqref{howisD2q} implies that
$q|_L<0$ outside of the origin.
Then we show that this is incompatible with having a sequence of singular points $x_k$ converging to $0$ (the reason for this incompatibility is that these points would force $q$ to being nonnegative on $L$). Hence we conclude that anomalous points are isolated for $n=3$.

Once this result is proved, by a Federer-type dimension reduction principle we handle the case $n \geq 4$ and prove that ${\rm dim}_{\mathcal H}(\Sigma_m^a)\leq m-1$. Note that Federer dimension reduction principle is not standard in this setting, the reason being that if $x_0$ and $x_1$ are two different singular points then the blow-ups at such points come from different functions, namely $u-p_{*,x_0}$ and $u-p_{*,x_1}$. 

Finally, the $C^{1,\log^{\epsilon_0}}$ regularity of $\Sigma_m$ comes as a simple consequence of our analysis combined with Caffarelli's asymptotic convexity estimate \cite{C77}.

\item[$\bullet$] {\it The case $m=n-1$.}
In this case we observe the following:
since $u-p_{*,0}=u\geq 0$ on $L$, then
$\tilde w_r|_L\geq 0$ for any $r>0$.
Hence, since ${\rm dim}(L)=n-1,$  by a trace inequality we deduce that $q\geq 0$ on $L$.
 Thus we have obtained that:\\
 - $q$ is $\lambda_*$-homogeneous;\\ 
- $\Delta q\leq 0$, $q\Delta q\equiv 0$, $q|_L\geq 0$;\\ 
- $\Delta q$ is supported on $L$.

We now note that this system is simply the PDE characterization of global homogeneous solutions to the so-called 
 {\it Signorini problem}, also known as {\it thin obstacle problem} (see for instance \cite{AC04,ACS08,GP09,FS17}).
In particular, since all global homogeneous solutions are classified when $n=2$,  we deduce that
in two dimensions
$$
\lambda_*\in \{2,3,4,\ldots\}\cup
\biggl\{4-\frac12,6-\frac12,8-\frac{1}{2},\ldots\biggr\}.
$$
Also, using again \eqref{howisD2q} and the fact that now $q|_L\geq 0,$ we can 
rule out $\lambda_*=2$. 
This proves that $\lambda_*\geq 3$ for $n=2$.
Then the higher dimensional case is handled  using again a Federer-type dimension reduction principle.

Furthermore, to show that $\Sigma_{n-1}$ is contained in a $C^{1,\alpha_0}$-manifolds, we prove that $\lambda_*\geq 2+\alpha_0>2$ in every dimension.
\end{itemize}

Finally, the $C^2$ regularity in two-dimensions requires a further argument based on a new monotonicy formula of Monneau-type.

\section{Future directions}
Although the results presented in the previous sections provide a very good understanding of the free boundary regularity in the classical obstacle problem, there are still several directions that are worth being investigated.

First of all, Theorem \ref{thm:main} shows that if $x_0 \in \Sigma_m^a$ with $m \leq n-2$, then the expansion 
\eqref{eq:C2 u} is optimal. On the other hand, we have seen that \eqref{eq:C2a u} holds with $\alpha=1$ at generic points.
Hence a natural question is whether, at a generic point $x_0$, there exists a unique third order harmonic polynomial $p_{3,x_0}$ such that 
$$
\|u(x_0+\cdot)-p_{*,x_0}-p_{3,x_0}\|_{L^\infty(B_r)}= o(r^3).
$$
As shown in \cite{FSell}, this is true at almost every generic point. So one may ask whether one can further improve the error $o(r^3)$ to $O(r^4)$, and more in general whether one can prove a Taylor expansion up to every order (at least at most points).
Motivated by applications to the Schaeffer's conjecture (which states that, for generic obstacles, the set of singular points should be empty \cite{Sch76}), in \cite{FRS} we prove such an expansion up to order 5. Then, as a corollary, we obtain the validity of Schaffer's conjecture in dimension 3 (the two dimensional case had already been proved by Monneau in \cite{M03}).
\smallskip

An important direction related to the discussion above is whether an expansion holds up to every order.
More in general, given a positive function $h$ of class $C^{k,\alpha}$ and a solution of $\Delta u=h\chi_{\{u>0\}}$ (this corresponds to the obstacle being of class $C^{k+2,\alpha}$, see \eqref{eq:obst var}), one may ask whether an expansion of order $k+2$ holds at most singular points.
\smallskip

On a different direction, we may note that all the results obtained up to now concern only the structure of regular and singular points when seen as ``disjoint sets''. For instance, Theorem \ref{thm:dico} does not say anything about the regularity of regular points as they approach the singular set.
Using complex variable techniques, an answer to this problem has been given for analytic obstacles when $n=2$ \cite{Sak91,Sak93} (see also \cite{CR76,M03} for some results in the case of smooth obstacles). It would already be very interesting to have a complete description of the free boundary in the two dimensional case for obstacles that are smooth (say $C^\infty$) but not analytic.
\smallskip

Finally, all these questions can be asked in the parabolic version of the obstacle problem, namely when $u$ solves\footnote{We note that the parabolic obstacle problem is related to the one-phase Stefan problem. The latter aims to describe the temperature distribution in a homogeneous medium undergoing a phase change, typically the melting of a body of ice maintained at zero degrees centigrade. Given are the initial temperature distribution of the water $\theta(t,x)\geq 0$, the heat of fusion (i.e., the amount of energy per unit 
mass needed to melt
the solid) that is usually assumed to be identically equal to $1$, and the energy contributed to the system through the boundary of the domain. The unknowns are the temperature distribution of the water as a function of space and time, and the  ice-water interface (i.e., the {free boundary}).
The relation between these two problems is given by the so-called Duvaut's transformation $u(t,x):=\int_0^t\theta(s,x)\,ds$ \cite{Duv,Duv2}. With this change of variables one can note that $\{\theta>0\}=\{u>0\}$, so any result on the free boundary regularity for the parabolic obstacle problem implies the validity of the same result also 
for the Stefan problem.}
$$
\partial_tu-\Delta u=-\chi_{\{u>0\}},\qquad u \geq 0,\qquad\partial_t u \geq 0.
$$
In \cite{FRS} we generalize Theorem \ref{thm:main} to this setting and, as a consequence, for $n \leq 3$ we prove that the free boundary is smooth outside of a closed set of ``singular times'' of dimension at most $1/2$.

\end{document}